\documentclass[12pt,reqno]{amsart}

\usepackage{amscd}
\usepackage{amssymb,amsmath}
\usepackage[latin1,utf8]{inputenc}
\usepackage{verbatim}
\usepackage{graphics}
\usepackage{xypic}

\usepackage{tikz}
\usetikzlibrary{shapes}

\usepackage{graphicx}

\reversemarginpar

\def\Z{\mathbb{Z}}

\def\ZZ{\Z \oplus \Z}
\def\R{\mathbb{R}}
\def\C{\mathbb{C}}

\def\SS{\mathbb{S}}
\def\KB{\mathbb{K}}
\def\PP{\mathbb{P}}
\def\I {\mathbb{I}}
\def\TT {\mathbb{T}}
\def\BB {\mathbb{B}}

\def\P{\pi_{1}}
\def\a{\alpha}
\def\b{\beta}
\def\g{\gamma}

\def\f{\phi}
\def\vf{\varphi}

\def\s{\sigma}
\def\w{\omega}
\def\G{\Gamma}
\def\lra {\longrightarrow}

\def\cal{\mathcal}

\def\tore {\cal{T}}
\def\sphere {\cal{S}}

\def\W {\Omega}

\def\ol{\overline}
\def\tl{\widetilde}

\def\M {\mathcal{M}}
\def\N {\mathcal{N}}
\def\F {\mathcal{F}}
\def\B {\mathcal{B}}
\def\C {\mathcal{C}}

\def\D {\mathcal{D}}
\def\A {\mathcal{A}}
\def\piece {\mathcal{Q}}
\def\V {\mathcal{V}}
\def\WW {\mathcal{W}}

\def\X {\mathbf{X}}
\def\Mg {\mathbf{M}}
\def\Ng {\mathbf{N}}
\def\pp {\mathbf{p}}
\def\p {\mathrm{p}}

\def\e {\mathrm{e}}
\def\v {\mathrm{v}}
\def\t {\mathrm{t}}
\def\x {\mathrm{x}}

\def\pK {\mathbb{Z}\rtimes \mathbb{Z}}

\newtheorem{lem}{Lemma}[section]

\newlength{\ecart}
\newlength{\largeur}
\title{The conjugacy problem in groups of non-orientable 
3-manifolds}

\begin{document}

\maketitle
\begin{center}
{\sc Jean-Philippe PR\' EAUX} 
\footnote{Laboratoire d'Analyse, Topologie et Probabilit\'es, UMR 7353, Aix-Marseille Universit\'e, 39 rue F.Joliot-Curie, F-13453
marseille
cedex 13\\
\indent {\it E-mail:} \ preaux@cmi.univ-mrs.fr\smallskip\\
\indent \begin{minipage}{12.1cm}{\sl Keywords:} \ Conjugacy problem, fundamental group,
non-orientable 3-man\-ifolds, topological decomposition of 3-manifolds, graphs of groups.
\end{minipage}%
}\smallskip\\
\end{center}


\begin{abstract}
We  prove that fundamental groups of non-orientable 3-manifolds have a solvable conjugacy problem
and construct an algorithm. Together with our earlier work on the conjugacy problem in groups of 
orientable geometrisable 3-manifolds, all  $\pi_1$ of  (geometrisable) 3-manifolds have a
solvable conjugacy problem. 
As corollaries, both the twisted conjugacy problem  in closed surface groups and the
 conjugacy problem in closed
surface-by-cyclic groups, are solvable.
\end{abstract}

\pagestyle{myheadings}
\markboth{Jean-Philippe PR\'EAUX}{The conjugacy problem in $\pi_1$ of non-orientable 3-manifolds} \markright{{THE CONJUGACY PROBLEM IN $\pi_1$ OF NON-ORIENTABLE 3-MANIFOLDS}}

\section*{Introduction}
Since formulated by M.Dehn in the early 1910's, the word and conjugacy  problems for finitely presented groups
have become fundamental in combinatorial group theory. Following the work of Novikov \cite{novikov} and
further authors on their general unsolvability, it has become quite natural to ask for any finitely presented group
whether it admits a solution or not. For example in \cite{dehn1,dehn2,dehn3}, Dehn has solved those
 problems in fundamental groups of closed surfaces; his motivation was coming from their
topological study.

 Given a finite
presentation of a group $G$, a solution to the
 {\sl word problem}  is an algorithm, which given two elements
  $\w,\w'\in G$ as words in the generators and their inverses,
   decides whether $\w=\w'$ in $G$ other not. A solution to the {\sl conjugacy problem}
    is an algorithm, which given $\w,\w'\in G$, decides whether
     $\exists\ h\in G$ such that $\w'=h\w h^{-1}$ in $G$ other not.
It turns out that the existence of a solution for $G$ does not depend on the finite presentation involved. 
We say that $G$ has a {\sl
solvable} word (resp. conjugacy)  problem  if  $G$ admits a solution to the word  (resp. conjugacy) problem.
\smallskip\\ \indent
By a 3-manifold we mean a connected compact manifold of dimension 3 with boundary; a 3-manifold may be orientable or
not. We work in the PL category; by the hauptvermutung and Moise's Theorem this is not restrictive.
According to the work of Thurston (cf. \cite{th}) an oriented 3-manifold $\M$ is {\sl geometrisable} if all the pieces obtained
in its canonical topological decomposition
 (roughly speaking along essential spheres, discs and tori) have an interior
that admits a complete locally homogeneous Riemannian metric.
 In the following we will say (abusively) that a  non-orientable 3-manifold is {\sl geometrisable} whenever the total space of its orientation cover 
is geometrisable; it is worth noting that this definition is weaker than the usual one conjectured for all  non-orientable 3-manifold\footnote{For short that  the pieces obtained in their topological decomposition are two-fold covered by punctured, orientable, complete locally homogeneous Riemannian 3-manifolds whose cover involutions are isometries. See \cite{scott} pages 484--485 for details.}. It's a deep result that all 3-manifolds turn out to be geometrisable;  the work
of Perelman  in the early 2000's, together with clarifications by 
several authors (cf. \cite{boileau}) proves this statement. 
The reader who might not feel comfortable with Perelman's proof may consider this assumption
 as an hypothesis on the non-orientable 3-manifolds involved.

In fundamental groups of geometrisable 3-manifolds, the word problem is  known to be 
solvable
since the work of Epstein and Thurs\-ton  on automatic group theory
(cf. \cite{epstein}).
We have proved in \cite{cp3mg} that  all groups of orientable geometrisable 3-manifolds have a solvable 
conjugacy problem; we will make a heavy use of these last two results in our proof.
 We focus on non-orientable 3-manifolds and we construct an algorithm that solves the conjugacy problem
  for their fundamental groups. 
  Therefore, all 3-manifolds have a fundamental group with solvable conjugacy problem, which 
  contrasts with higher dimensions. 
  We also state as corollaries that the conjugacy problem in surface-by-cyclic groups 
  and that the twisted conjugacy problem in surface groups are both solvable (cf. \S 1).

The solution for groups of non-orientable 3-manifolds does not arise from a solution in the oriented case, since D. Collins and
C. Miller have shown that the conjugacy problem can be unsolvable in a group even when 
solvable in an index 2 subgroup
(\cite{collins}).
 Nevertheless
our strategy will consist essentially in reducing as far as possible to the conjugacy problem in the orientation covering space.

We briefly emphasize two points that sound noteworthy to us. 
On the one hand the core of the algorithm makes use of basic solutions to the word and conjugacy 
problems in groups of orientable 3-manifolds, themselves reducing to basic solutions,  built from 
biautomatic group theory, in groups of the basic pieces, Seifert fiber spaces and finite volume 
hyperbolic manifolds, and does not require any naive enumerative algorithm. Therefore, one may expect to 
deduce an efficient algorithm. On the other hand our general strategy that reduces the problem to the orientation covering space may sometimes succeed in resolving the conjugacy problem in a group $G$ 
 containing an index 2 subgroup $H$ with solvable conjugacy problem; see Theorem F.

\section{Statement of the results}
 This work is mainly devoted to prove the following result:\medskip\\
\noindent{\bf Theorem A.\ [Main result.]}\; {\it The conjugacy problem is solvable in fundamental groups of non-orientable geometrisable 3-manifolds.}\medskip\\
 Together with a solution in the
oriented case (cf. \cite{cp3mg}), one obtains:\medskip\\
\noindent{\bf Theorem B.}\; {\it The conjugacy problem is solvable in fundamental groups of geometrisable 3-manifolds.
Topologically rephrased, given any pair of loops $\gamma$, $\gamma'$ in a geometrisable 
3-manifold, one can decide whether they are freely homotopic.}\medskip

Note that Theorem A does not follow as a corollary of the oriented case; indeed the usual 
technique that
consists in translating the problem to the oriented cover fails to yield a solution, 
since the conjugacy problem
can be unsolvable in a group even when solvable in an index 2 subgroup (cf. \cite{collins}).
%
%
It is worth noting that we do not only show that a solution exists but rather give a constructive process to 
build an
algorithm; moreover when applied to conjugate elements $u$ and $v$ the algorithm implicitly
 produces a conjugating element $h$, {\sl i.e.} such that $u=hvh^{-1}$.\smallskip

 Theorem B has several consequences. 
 First the generalized word and conjugacy problems relative to boundary 
 subgroups are both solvable. More precisely:
 \medskip\\
\noindent{\bf Theorem C.}\; {\it Let $\M$ be a  geometrisable 3-manifold and $\F\subset\partial \M$ a compact
connected surface. Denote by $G=\pi_1(\M)$ and $H=i_*(\pi_1(\F))$; there exists algorithms
 that  decide for any $g\in G$ respectively
whether $g\in H$ and whether $g$ is conjugate to an element of $H$.
Topologically rephrased given any loop $\g$ (resp. any $*$-based loop for some $*\in \F$) in $\M$ 
one can decide whether up to homotopy (resp. $*$-fixed homotopy) $\g$ lies in $\F$.\medskip
}

\noindent{\it Proof that Theorem B $\implies$ Theorem C.}
 Double the
3-manifold $\M$ along the identity on $\F$ to obtain  the 3-manifold $\M\sqcup_{\F}\M$. 
The proof of lemma 1.2 of \cite{cp3mg}, as well as the
observation that the orientation cover of $\M\sqcup_{\F}\M$ is the double of the orientation cover
 of $\M$ along the lift(s) of $\F$
show that $\M\sqcup_{\F}\M$ is geometrisable. Its group splits into an amalgam $\G=G*_H G$ of 
two copies of $G=\P(\M)$ along the identity of $H=i_*(\P(\F))$.
 Given $g$ lying in the $G$-left factor we denote by $\bar{g}$ the 
corresponding element in the $G$-right
factor. Since the gluing map is the identity, for any  $h\in H,\, h=\bar{h}$, and one obtains by applying elementary
 facts upon amalgams (cf. Corollary 4.4.2 and Theorem
4.6 in \cite{mks})
 that $g,\bar{g}$ are equal (resp. conjugate) in $\G$ if and only if $g\in H$ (resp. $g$
is conjugate in $G$ to some $h\in
H$). Hence with a solution to the word problem (resp. conjugacy problem) in $\G$
provided by Theorem B, it suffices to decide whether $g=\bar{g}$ (respectively $g$ and $\bar{g}$ 
are conjugate).\hfill $\square$\\

A second consequence concerns the conjugacy problem in surface by cyclic groups. It has been
 proved in \cite{freebycyclic} that (f.g.
free)-by-cyclic groups have a solvable conjugacy problem. As a complementary
 result we can deduce
from Theorem B together with the Dehn-Nielsen Theorem the same statement concerning 
(closed surface)-by-cyclic groups,
so that any (compact surface)-by-cyclic group turns out to have a solvable conjugacy problem.\\

\noindent{\bf Theorem D.}\; {\it The conjugacy problem is solvable in closed surface-by-cyclic groups.}\smallskip

\noindent {\it Proof that Theorem B $\implies$ Theorem D.} Let $\F$ be a closed surface, $K=\pi_1(\F)$, and let $G$ be an
extension of $K$ by a cyclic group $C$. In case $C$ is finite, $G$ is biautomatic (follows from  \cite[Definition 2.5.4 and Example 2.3.6]{epstein}) and hence has a 
solvable conjugacy problem (cf. \cite{gs} or \cite[Theorem 2.5.7]{epstein}). In case $C$ is infinite, the extension splits as
 $G=K\rtimes_\f\Z$ for some
$\f\in Aut(K)$. The Dehn-Nielsen Theorem (cf. \cite[Theorem 3.4.6]{dehn-nielsen}) shows that $\f$ is 
induced by an homeomorphism $f$ of the surface $\F$ so that $G$ is isomorphic to
 the fundamental group of the bundle over
$\SS^1$ with fiber $\F$ and sewing map $f$. It follows from the Thurston geometrisation Theorem 
(\cite{th}) that such a bundle is geometrisable and finally Theorem B shows that $G$ has a solvable 
conjugacy problem.\hfill$\square$\\

A third consequence concerns the twisted conjugacy problem in surface groups. 
Given a group $G$ and an automorphism
$\f$ of $G$ the twisted conjugacy problem is solvable in $(G,\f)$ if one can algorithmically decide 
given any
$u,v\in G$ whether there exists $g\in G$ such that $\f(g)ug^{-1}=v$. The twisted conjugacy problem is said to be solvable in $G$
when solvable in $(G,\f)$ for any automorphism $\f\in Aut(G)$.
\\

\noindent{\bf Theorem E.}\; {\it The twisted conjugacy problem is solvable in closed surface groups.}\smallskip\\
Together with the case of f.g. free groups (cf. \cite{freebycyclic}) the same result holds for fundamental groups of compact surfaces.\smallskip\\
{\it Proof that Theorem D $\implies$ Theorem E.} Let $\F$ be a closed surface,
$K=\pi_1(\F)$, $\f$ an automorphism of $K$
and $G=K\rtimes_\f\Z$ given by the presentation $=<G,t\,|\,\forall\,g\in G,\, tgt^{-1}=\f(g)>$. Given $u,v\in K$, there exists $g\in K$ such that $\f(g)ug^{-1}=v$
 if and
only $\exists\, g\in K$ such that $g\,t^{-1}u\,g^{-1}=t^{-1}v$ if and only if $t^{-1}u$ is conjugate to
$t^{-1}v$ in $G$: indeed if there exists $ht^p$, with $h\in K$ and  $n\in\Z$, that conjugates $t^{-1}u$ into $t^{-1}v$, then so 
does
 $h_n=ht^p(t^{-1}u)^n$ for any $n\in\Z$, and in particular for $h_p$, which belongs to $K$. Hence a solution to the conjugacy
 problem in $G$ yields a solution to the twisted conjugacy problem in $(K,\f)$.\hfill$\square$\\
 
 We conclude with a general result  inspired by our strategy that may help sometimes for solving the conjugacy problem in a group $G$ containing an index 2 subgroup with solvable conjugacy problem.\smallskip
 
 \noindent{\bf Theorem F.}\; {\it Let $G$ be a finitely presented group and $H$ an index 2 subgroup of $G$ given by a right coset function of $G\mod H$. Suppose that:
 \begin{itemize}
 \item[(i)] $H$ has a solvable conjugacy problem,
 \item[(ii)] for any couple of order two elements $u,v\in G$ one can decide whether $u,v$ are conjugate in $G$, and
 \item[(iii)] for any element $u\in G$ with order $>2$, one can decide for any couple of elements of $G$ lying in the centralizer $Z_G(u^2)$ whether they are conjugate in $Z_G(u^2)$.
  \end{itemize}
  Then $G$ has a solvable conjugacy problem.
 }
 
 \begin{proof}
 Since $H$ is an index 2 subgroup of $G$ having a solvable conjugacy problem, then both $H$ and $G$ have a solvable word problem. Note also that the right coset function allows to decide for any element in $G$ whether it lies in $H$ and yields a Reidemeister rewriting process for $H$ (\cite[\S 2.3]{mks}). Let $u,v\in G$ be given; one can suppose that they are both non-trivial. Check whether $u^2=1$ and $v^2=1$.  If both occur then use $(ii)$ to decide whether $u$ and $v$ are conjugate in $G$. If exactly one of $u,v$ has order 2 then they are definitely non-conjugate. So suppose in the following  that $u,v$ both have order greater than 2.
 Let $a\in G\smallsetminus H$ be arbitrary and $\bar{v}=ava^{-1}$.
 Use $(i)$ to decide whether $u^2$ is conjugate in $H$ either to $v^2$  or to $\bar{v}^2$. 
Suppose that one case occurs for otherwise $u,v$ are non-conjugate in $G$, and find $h\in H$ such that $u^2=hv^2h^{-1}$ (resp. $u^2=h\bar{v}^2h^{-1}$).  Then $v'=hvh^{-1}$ (resp. $h\bar{v}h^{-1}$) lies in $Z_G(u^2)$, and $u,v$ are conjugate in $G$ if and only if $u,v'$ are conjugate in $Z_G(u^2)$, that one can decide using $(iii)$. 
 \end{proof}

\section{Proof of Theorem A}

We now turn to the proof of the main result.  
 We start from an arbitrary non-orientable (geometrisable) 3-manifold $\M$ given by a
triangulation and we construct an algorithm that solves the conjugacy problem in $\P(\M)$.   

The process is done in four
steps.
In the first step  we reduce to the  closed irreducible case,
{\it i.e.} we prove that one recovers solutions in groups of non-orientable 3-manifolds from solutions in groups of closed irreducible 3-manifolds.
In the second step we construct the orientation cover $p:\N\lra\M$ of $\M$ and of its cover involution $\s:\N\lra\N$ as well as the topological decompositions of $\M$ and $\N$; that of $\N$ is obtained from known algorithms for the decomposition of orientable 3-manifolds derived from the Haken's theory of normal surfaces (\cite{jt,jlr}); then deforming until it becomes 'almost' $\s$-invariant yields that of $\M$.
In the third step we construct the induced splittings  of $\pi_1(\M)$ and $\pi_1(\N)$ as graphs of groups $\Mg$, $\Ng$ and the covering of graphs of groups $\pp:\Ng\lra\Mg$, which are useful to deal with elements in $\pi_1(\M)=\pi_1(\Mg)$ and $\pi_1(\N)=\pi_1(\Ng)$; we also establish all the basic algorithms needed in the final and fourth step, where is given the core of the algorithm.

\subsection{\bf Step 1: Reduction to the closed irreducible case.}
The preliminary step reduces the proof to the case of closed irreducible geometrisable 3-manifolds.

\begin{lem}\label{reduction}
The conjugacy problem in groups of non-orientable geom\-etrizable 3-ma\-nifolds reduces to that of 
closed irreducible geometrisable 3-manifolds. Moreover given a triangulation of $\M$  the reduction process is constructive.
\end{lem}

\noindent{\it Proof.} Let $\M$ be a non-orientable geometrisable 3-manifold; we focus on the conjugacy problem
in $\pi_1(\M)$. The reduction process is made in two steps, 
by first reducing to closed non-orientable 3-manifolds and then to closed irreducible
3-manifolds.

Gluing a 3-ball to each spherical component of the boundary $\partial \M$ leaves
$\pi_1(\M)$ unchanged; so we suppose in the following that $\M$ has no spherical
boundary component. If $\partial \M$ is non-empty, double $\M$ along its
boundary to obtain the closed non-orientable 3-manifold that we denote by $2\M$.
Lemma 1.1 of \cite{cp3mg} asserts that the inclusion map  $i: \M\lra 2\M$
induces a monomorphism $i_*:\pi_1(\M)\lra\pi_1(2\M)$, and that
$u,v\in\pi_1(\M)$ are conjugate in $\pi_1(\M)$ if and only if $i_*(u)$ and
$i_*(v)$ are conjugate in $\pi_1(2\M)$; hence the conjugacy problem in
$\pi_1(\M)$ reduces to that in $\pi_1(2\M)$: to decide whether
$u,\,v$ are conjugate in $\pi_1(\M)$ it suffices to decide whether
$i_*(u),\,i_*(v)$ are conjugate in $\pi_1(2\M)$. 

Let's prove that the closed 3-manifold $2\M$ is 
geometrisable (without reference to Thurston-Perelman).
Notice first that if $\ol{\M}$ and $\ol{2\M}$ denote respectively the orientation 
covering space of $\M$ and of $2\M$, then $\ol{2\M}$ is the double space  $2\ol{\M}$ of $\ol{\M}$: 
indeed, the orientation cover $\ol{\M}\lra\M$ naturally extends to a 2-fold cover 
$2\ol{\M}\lra 2\M$ of their double spaces, therefore since $2\ol{\M}$ is orientable, by unicity of the orientation cover,  $2\ol{\M}$ and $\ol{2\M}$ are homeomorphic. 
Since $\M$ is
geometrisable so is $\ol{\M}$ (recall our weak version of the definition of being geometrisable for a non-orientable 3-manifold).
Now Lemma 1.2 of \cite{cp3mg} states that the double space of an orientable geometrisable 3-manifolds is also geometrisable, and hence, with the above
$\ol{2\M}$ is geometrisable; therefore the non-orientable 3-manifold ${2\M}$ is geometrisable. 

It proves that the conjugacy
problem in groups of geometrisable non-orientable 3-manifolds reduces to that for  those which are closed. Moreover the reduction process is constructive since from a triangulation of $\M$  one can effectively
produce a triangulation of $2\M$ and the monomorphism $i_*: \pi_1(\M)\lra \pi_1(2\M)$.

We now turn to the second step; so suppose furthermore that  $\M$ is
closed. A Kneser-Milnor decomposition splits $\M$ into a connected sum of the
prime closed geometrisable (possibly orientable) factors $\M_1, \M_2,\ldots
,\M_n$, inducing a split of $\pi_1(\M)$ into the free product of 
$\pi_1(\M_1)$, $\pi_1(\M_2),\ldots,\pi_1(\M_n)$. A basic fact on conjugacy in
free products (cf. \cite[Theorem 4.2, \S 4.1]{mks}) shows that the conjugacy
problem in $\pi_1(\M)$ reduces to those in
the $\pi_1(\M_i)$'s. Now either $\M_i$ is an $\SS^2$-bundle over $\SS^1$ and $\pi_1(\M_i)\simeq\Z$, or $\M_i$ is irreducible. Hence the conjugacy problem in
$\pi_1(\M)$ reduces to that in groups of closed irreducible
geometrisable 3-manifolds. When $\M$ is given by a triangulation, an algorithm
appearing in \cite{jt} for the Kneser-Milnor decomposition allows to perform the
reduction constructively.\hfill$\square$\smallskip

 In the achievement of the proof, the Lemma \ref{reduction} 
 leads us to the case of groups of closed irreducible geometrisable non-orientable 
3-manifolds. \smallskip

{\noindent\sl$\bullet$\ \begin{minipage}[t]{\largeur}
{ In the following $\M$ denotes a closed irreducible geometrisable non-orientable  3-manifold, with orientation cover 
$p:\N\longrightarrow \M$.}
\end{minipage}
}

%

\subsection{\bf Step 2: Algorithms for the topological decompositions of $\M$ and $\N$.} \label{section:topological}
We start from a closed irreducible
non-orientable geometrisable 3-manifold $\M$ given by a triangulation. We show how one algorithmically constructs the orientation cover $p:\N\lra\M$ and appropriate topological decompositions of $\M$, $\N$.
\begin{figure}[h]
\centerline{\includegraphics[scale=0.8]{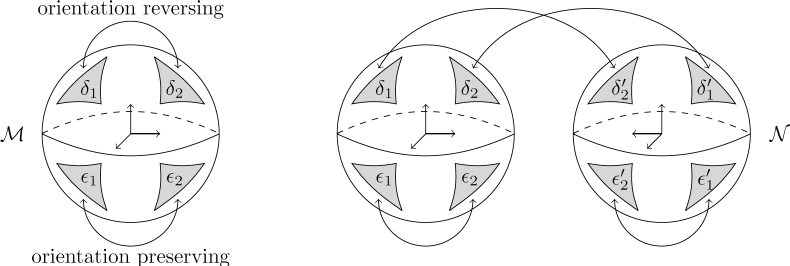}}
\caption{\small Construction of $\N$ from $\M$ given as a PL-ball with identifications on the boundary.}
\label{figcover}
\end{figure}
\begin{lem}[Algorithm $\mathfrak{Top1}$  for  the orientation cover]\label{l4}
Given a triangulation of $\M$ one can algorithmically produce a
triangulation of its orientation covering space $\N$ as well as the covering
map $p:\N\longrightarrow\M$ and the covering involution $\s:\N\longrightarrow\N$.
\end{lem} 
\begin{proof}
 The triangulation of  $\M$ can be easily given as a triangulation of a PL-ball $\B$ together with a
gluing of pairs of triangles in $\partial \B$. Pick an orientation of $\B$; it induces an orientation on each triangle in
$\partial \B$. Identify paired triangles in $\partial \B$ each time their gluing preserves orientation, to obtain a new
oriented PL-manifold $\C$ together with orientation reversing gluings  of pairs of triangle in $\partial \C$. Consider
a copy $\C'$ of $\C$ and endow $\C'$ with the reverse orientation. For each triangle $\delta$ in $\partial \C$ denote by $\delta'$ its copy
in $\partial \C'$. For each gluing of triangles $\delta_1,\delta_2$ in $\partial \C$, glue coherently in $\C\cup \C'$, $\delta_1$ with
$\delta_2'$ and $\delta_1'$ with $\delta_2$ (cf. figure \ref{figcover}). The manifold obtained is the covering space $\N$ together with a triangulation, and the construction
implicitly produces the covering map $p:\N\longrightarrow \M$ as well as the covering involution $\s: \N\lra\N$.
\end{proof}
{\noindent\sl$\bullet$\ \begin{minipage}[t]{\largeur}
{ Apply Algorithm $\mathfrak{Top1}$ to construct a triangulation of the 2-fold covering space $\N$ of $\M$ as well as the covering projection $p$ and
involution $\s$.}
\medskip
\end{minipage}
}

%
\indent The oriented manifold $\N$ may be reducible. In such a case (cf. \cite{swarup}) $\N$ contains a compact surface $\Sigma$ whose 
components are $\s$-invariant essential spheres such that:
\begin{itemize}
\item[--] $\N$ cut along $\Sigma$ decomposes  into $\s$-invariant components:
$\N_1$, $\N_2,\ldots ,\N_p$ and each manifold $\widehat{\N}_i$
obtained by filling up all $\SS^2\subset
\partial \N_i$ with  balls  is irreducible and non-simply connected;\smallskip
\item[--] let $n$ be the number of non-separating components in $\Sigma$;
$\pi_1(\N)$ decomposes as a free product of $\pi_1(\N_1)$,$\ldots\pi_1(\N_p)$ and of a free group with rank $n$: $\pi_1(\N)\simeq \pi_1(\N_1)*\cdots *\pi_1(\N_p)*F_n$;\smallskip
\end{itemize}
and $\Sigma$ has
 image in $\M$  a compact surface $\Pi=p(\Sigma)$ whose components are two-sided  projective planes
 such that:
\begin{itemize}
\item[--] $\M$ cut along $\Pi$ has components
$\M_1,\M_2,\ldots, \M_p$  where the
covering projection sends each $\N_i$ onto $\M_i$;\smallskip
\item[--] it induces a splitting of $\pi_1(\M)$ as a graph of groups whose vertex groups are $\pi_1(\M_1),\pi_1(\M_2),$
$\ldots ,\pi_1(\M_p)$  and all edge groups have order 2.\smallskip
\end{itemize}

\begin{lem}[Algorithm $\mathfrak{Top2}$ for coherent decompositions along $\SS^2$ of $\N$ and $\PP^2$ of $\M$]\label{l3}
One can algorithmically find systems of pairwise disjoint essential $\s$-invariant spheres $\Sigma$ in $\N$ and
projective planes $\Pi$ in $\M$ as above.
\end{lem}
\begin{proof}
 Apply \cite[Algorithm 7.1]{jt} (or, for a bound on complexity, an improved algorithm in \cite{jlr}),  to find, if any, an essential sphere
$\sphere_0$ in $\N$. If none exists then $\Sigma=\Pi=\varnothing$ and the process stops. Otherwise apply to $\sphere_0$ the
following classical argument (cf. \cite{tol}) to construct a $\s$-invariant essential sphere $\sphere$ in $\N$.
Since $\M$ is irreducible then $\s \sphere_0\cap \sphere_0\not=\varnothing$ for otherwise $p(\sphere_0)$ would be an essential sphere in $\M$. If $\s \sphere_0=\sphere_0$ there is nothing left to prove so suppose it does not occur.
 Deform slightly $\sphere_0$ so that $\s \sphere_0\cap \sphere_0$ consists of $n>0$ closed simple curves. Consider  such a
curve, which in addition bounds an innermost disk $\D$ in $\s \sphere_0$ ({\sl i.e.} a disk in $\s \sphere_0\setminus (\s \sphere_0\cap
\sphere_0)$), as well as a disk $\D'$ in $\sphere_0$. Consider the two spheres  $\sphere_1=\sphere_0\cup \D\setminus int(\D')$ and $\sphere_2=\D\cup \D'$ and
perform small isotopies (see figure \ref{figsmallisotopies}) so that $\sphere_i=\s \sphere_i$ or $\sphere_i\cap \s \sphere_i$ has fewer than $n$ components
($i=1,2$).
\begin{figure}[h]
\centerline{\includegraphics[scale=0.8]{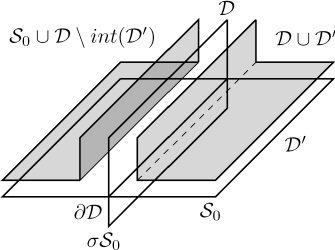}}
\caption{\small By considering small enough collar neighborhoods $N(\sphere_0)$ of $\sphere_0$ and $\s N(\sphere_0)$ of $\s\sphere_0$ in a subdivision of the triangulation of $\N$, one can deform by isotopy the spheres $\sphere_1=\sphere_0\cup\D\setminus int(\D')$ and $\sphere_2=\D\cup\D'$ such that either $\s(\sphere_i)=\sphere_i$ or they become transverse and $\sphere_i\cap\s\sphere_i$ has fewer components, $i=1,2$.}
\label{figsmallisotopies}
\end{figure}

 Moreover at least one of $\sphere_1,\sphere_2$ does not bound a ball in $\N$: for suppose on the contrary that $\sphere_1,\sphere_2$ bound respective balls $\B_1,\B_2$ then either: $(i)$ $\B_1\supset \sphere_2$ and $\sphere_0$ also bounds a ball included in $\B_1$, or $(ii)$ $\B_1\cap \B_2=\varnothing$ and there exists a collar neighborhood $N(\D)$ of $\D$ such that $\B_1\cup\B_2\cup N(\D)$ is a ball bounded by $\sphere_0$; it would contradict that $\sphere_0$ is essential. Apply the $\SS^3$ recognition algorithm (cf. \cite{rubinstein}) to each component of $\N\setminus \sphere_1$, $\N\setminus \sphere_2$ with one ball glued on the boundary, to check
which of $\sphere_1,\sphere_2$, say $\sphere_1$, does not bound a ball. Then apply the same process to $\sphere_1$ instead of $\sphere_0$, and so
on. Since the number of components of $\s \sphere_i\cap \sphere_i$ decreases it will finally stop, leading us with a $\s$-invariant
essential sphere $\sphere$ in $\N$.

Cut $\N$ along $\sphere$ and then glue balls $\B_1,\B_2$ to its boundary to obtain $\N_1$ and (possibly) $\N_2$. 
 Since $\s$  preserves $\sphere$ and
reverses the orientation both on $\N$ and $\sphere$ it necessarily preserves each component of $\N\setminus \sphere$. 
Restrict then extend
$\s$ to an involution of $\N_i$ with fixed points ($i=1,2$). Then apply the same argument as above to search for
essential $\s$-invariant spheres in $\N_1$ and $\N_2$. We furthermore need to deform each such sphere so that it lies in
$\N\setminus \sphere$, that is, so that it does intersect neither $\B_1$ nor $\B_2$.


Suppose without loss of generality that we have found an essential $\s$-invariant sphere $\sphere_0$ in $\N_1$ which
intersects $\B_1$. After slightly deforming $\sphere_0$ by isotopy, $\sphere_0\cap \partial \B_1$ consists of
$m>0$ simple closed
curves. Let $\D'$ be an innermost disk in $\sphere_0$; $\partial \D'$ bounds a disk $\D$ in $\partial \B_1$.
If $\D'$ lies in $\B_1$, consider $\sphere_1=\sphere_0\cup (\D\cup\s \D)\setminus int (\D'\cup\s \D')$ and deform it by a small isotopy  (cf. figure \ref{deform2})  so
that $\sphere_1$ becomes a $\s$-invariant essential sphere and $\sphere_1\cap\partial \B_1$ has less than $m$ components.

\begin{figure}[h]
\centerline{\includegraphics[scale=0.8]{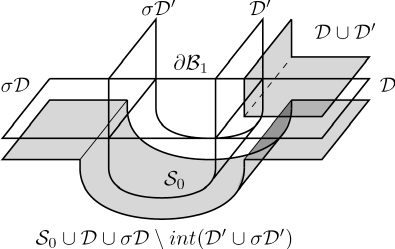}}
\caption{\small By considering small enough collar neighborhoods $N(\sphere_0)$ of $\sphere_0$ and $N(\partial \B_1)$ of $\partial\B_1$ in a subdivision of the triangulation of $\N$, one can deform by isotopy the spheres $\sphere_1=\sphere_0\cup (\D\cup\s \D)\setminus int (\D'\cup\s \D')$ and $\sphere_2=\D\cup \D'$ so that their numbers of intersection with $\partial\B_1$ decrease.}
\label{deform2}
\end{figure}

If $\D'$ does not lie in $\B_1$, consider the two spheres $\sphere_1$ as above and $\sphere_2=\D\cup \D'$ (cf. figure \ref{deform2}). At least one of $\sphere_1$, $\sphere_2$
does not bound a ball (for the same reason as above); use the $\SS^3$ recognition algorithm to check so. If $\sphere_1$ does not bound a ball, as above, after
a small isotopy, $\sphere_1$ is a $\s$-invariant essential sphere such that $\sphere_1\cap \partial \B_1$ has less than $m$
components. If $\sphere_2$ does not bound a ball, after a small isotopy, one obtains an essential sphere $\sphere_2$; but $\sphere_2$ may
 not be $\s$-invariant. Since $\M$ is irreducible one has $\sphere_2\cap \s \sphere_2\not=\varnothing$ and using the same procedure as
above one constructs a $\s$-invariant essential sphere $\sphere_3$, included in $\sphere_2\cup \s \sphere_2$ up to small isotopy, such
that $\sphere_3\cap\partial \B_1$ has less than $m$ components.

By applying this process as long as possible one finally obtains, if any, a $\s$-invariant essential sphere in $\N_1$ which
does not intersect $\B_1$ nor $\B_2$. Cut $\N_1$ along this sphere, glue balls and then apply the same process to the
manifolds obtained, while they do contain an essential sphere. According to the Kneser-Milnor Theorem it will finally
stop, leading us with the compact surface $\Sigma$ consisting of $\s$-invariant essential spheres in $\N$, as described above.
\end{proof}

{\noindent\sl$\bullet$\ \begin{minipage}[t]{\largeur}
{Apply Algorithm $\mathfrak{Top2}$ to find a surface $\Sigma$ in $\N$ made of  essential $\s$-invariant
spheres  and the surface $\Pi=p(\Sigma)$ in $\M$ made of two-sided projective planes.} 
\smallskip
\end{minipage}
}

 Cut $\N$ along $\Sigma$ and $\M$ along $\Pi$ so as to split them into pieces, respectively, $\N_1,\N_2,\ldots,\N_p$ and $\M_1,\M_2,\ldots ,\M_p$; the involution $\s$ restricts on each $\N_i$ to a free involution with quotient $\M_i$.
Consider 
$\widehat{\N}_1,\widehat{\N}_2,\ldots,\widehat{\N}_p$ obtained by filling up all spheres in $\partial \N_1, \partial \N_2,\ldots ,\partial \N_p$ with balls. 
Each involution $\s:\N_i\longrightarrow\N_i$ extends uniquely up to isotopy to an involution $\s:\widehat{\N}_i\longrightarrow \widehat{\N}_i$ with orbit space an orbifold $\widehat{\M}_i$ obtained from $\M_i$ by gluing a cone over $\PP^2$  on each projective plane in $\partial\M_i$.
\smallskip


Each manifold $\widehat{\N}_i$, $i=1 \ldots p$, is irreducible. Hence there exists a (possibly empty) 2-sided compact surface $\W_i\subset \widehat{\N}_i$ that is unique up to isotopy, and such that:
\begin{itemize}
\item[--] $\W_i$ is minimal with respect to  inclusion,
\item[--] components of $\W_i$ (if any) are essential tori, and 
\item[--] each component of $\widehat{\N}_i\setminus\W_i$ is either atoroidal or a Seifert fiber space;
\end{itemize}
 the so-called {\sl JSJ characteristic torus decomposition} (cf. \cite[Theorem 1.9]{hatcher} or \cite[Theorem 3.4]{bonahon}, see also \cite{js}). If $\W_i\not=\varnothing$, consider such a surface $\W_i$ that in addition satisfies:
\begin{itemize}
\item[--] $\W_i$ lies in $\N_i$, and\smallskip
\item[--] $p(\W_i)$ is a two-sided compact surface $\Xi_i$ in $\M_i$ whose components are essential tori and Klein bottles, and\smallskip
\item[--] whenever $\widehat{\N}_i$ is not a $\TT^2$-bundle over $\SS^1$ modeled on $Sol$ geometry: for each component $\tore$ of $\W_i$, if $\s\tore\not\in \W_i$ then $\tore$ and $\s\tore$ cobound in $\widehat{\N}_i$ a  component homeomorphic to  $\TT^2\times\I$ that is preserved under $\s$.
\end{itemize}
When $\W_i=\varnothing$, set $\Xi_i=\varnothing$. Such (possibly empty) surfaces $\W_i$ in $\N_i$ and $\Xi_i$ in $\M_i$ are what we call {\it coherent JSJ decompositions} of $\M_i,\N_i$.
%


\begin{lem}[Algorithm $\mathfrak{Top3}$ for the JSJ decompositions]\label{JSJ}
One can  algorithmically construct coherent  JSJ decompositions $\W_i\subset \N_i$ and $\Xi_i\subset \M_i$ ($i=1 \ldots p$).
\end{lem}

\begin{proof}
During the whole proof, we use notations $\widehat{\N}$, $\N$, $\W$, $\M$ and $\Xi$ instead of $\widehat{\N_i}$, $\N_i$, $\W_i$, $\M_i$ and $\Xi_i$.
Apply \cite[Algorithm 8.2]{jt} to find the JSJ decomposition $\W$ as well as the characteristic Seifert submanifold of
$\widehat{\N}$. Suppose in the following that $\W\not=\varnothing$ for otherwise there is nothing left to prove.

 Apply the same argument as in the proof of the last lemma to deform the tori in $\W$ so that they all lie
in the interior of $\N$.
Deform slightly $\W$ so that $\W\cap\s \W$ is either empty or consists of simple closed curves and of whole tori
$\tore_j$'s. Then deform  $\W$ so that each closed curve component in $\W\cap\s \W$ becomes essential in $\W$. For suppose that $\s \tore_1\cap
\tore_2$ has a non-essential closed curve as component, then it necessarily contains a curve bounding an innermost disk $\D$ in $\s
\tore_1$ such that $\partial \D$ bounds a disk $\D'$ in $\tore_2$. Replace in $\W$, $\tore_1$ by $\tore_1\cup \s \D'\setminus int(\s \D)$ and
$\tore_2$ by $\tore_2\cup \D\setminus int(\D')$, so that after a small isotopy the number of components in $\W\cap \s \W$ decreases (cf. figure \ref{deform4}).
Apply this process until each closed curve in $\W\cap\s \W$ becomes essential in $\W$. 

\begin{figure}[h]
\centerline{\includegraphics[scale=0.8]{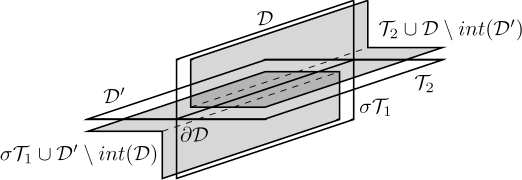}}
\caption{\small By considering small enough collar neighborhoods $N(\tore_2)$ of $\tore_2$ and $N(\s\tore_1)$ of $\s\tore_1$ in a subdivision of the triangulation of $\N$, one can deform by isotopy the tori $\tore_2\cup\D\setminus int(\D')$ and $\s\tore_1\cup\D'\setminus int(\D)$ so that their number of intersection decreases.}
\label{deform4}
\end{figure}

Let $\tore_j$ be a component of $\W$ such that $\s \tore_j\cap\W$ consists of simple closed essential curves. 
We prove first that $\s \tore_j\cap\W$ cannot consist of exactly one essential curve $\gamma$. For otherwise, let $\tore$ be a component of $\W$ such that $\s\tore_j\cap\tore=\gamma$; necessarily the torus $\tore$ is non-separating. Consider a regular neighborhood $V(\W)$ of $\W$ in $\N$ and  $\N'=\N\setminus int(V(\W))$; $\gamma$ gives rise to two disjoint essential curves $\gamma^-,\gamma^+$ in two different components of $\partial \N'$, both parallel to $\gamma$, and such that $\gamma^-,\gamma^+$ cobound an essential annulus in a component $\N''$ of $\N'$. Necessarily $\N''$ is a Seifert fiber space, with at least two boundary components. 
If $\N''\approx\TT^2\times\I$ then $\gamma^-,\gamma^+$ are regular fibers in a Seifert fibration of $\N''=\N'$, which extends to $\N$. This contradicts the minimality of $\W$.
If $\N''\not\approx \TT^2\times\I$, according to \cite[Lemma II.2.8]{js}, $\gamma^-,\gamma^+$ are homotopic to regular fibers of a fibration of $\N''$ and the same argument shows that $\W\setminus\tore$ is also a JSJ decomposition of $\N$, which once again contradicts the minimality of $\W$. This leads us to a contradiction.

Now let $\tore_j$ be a component of $\W$ such that $\s\tore_j\cap\W$ consists of at least two simple closed essential curves. The family of essential curves is pairwise disjoint so that they cut $\s\tore_j$ into annuli. Let $\gamma,\gamma'$ be two such curves cobounding an innermost annulus $\A$ ({\it i.e.} $\A$ intersects $\W$ into $\gamma,\gamma'$). Let $\tore,\tore'\subset\W$ be such that $\gamma\subset\tore$ and $\gamma'\subset\tore'$. Suppose that $\tore\not=\tore'$; consider as above $\N'=\N\setminus int(V(\W))$, it has a component $\N''$ which is a Seifert fiber space with $\gamma,\gamma'$ lying in different components of $\partial\N''$, and $\gamma,\gamma'$ cobound an annulus in $\N''$.

The case $\N''\approx\TT^2\times\I$ is discarded since $\W$ has at least two components. So with  \cite[Lemma II.2.8]{js} $\gamma,\gamma'$ are homotopic to regular fibers. Now let $\gamma''\in\s\tore_j\cap\tore''$ for some $\tore''\subset\W$, such that $\gamma,\gamma''$ cobound an innermost annulus $\B\not=\A$ in $\s\tore_j$. The same argument as above shows that $\N'$ contains $\N'''$ which is a Seifert fiber space and that if $\tore\not=\tore''$, $\gamma$ is homotopic to a regular fiber of $\N'''$. Hence  $\tore'\not=\tore\not=\tore''$ is impossible because otherwise the Seifert fibrations of $\N''$, $\N'''$ would both extend to $\N''\cup\N'''\cup int(V(\tore))$ and $\W\setminus \tore$ would be a smaller JSJ decomposition.

In summary when $\s\tore_j\cap\W$ consists of simple closed curves one can find $\tore\subset\W$ and two curves $\gamma,\gamma'\subset\tore$ cobounding in $\s\tore_j$ an innermost annulus $\A$. The curves $\gamma,\gamma'$ cobound also an annulus $\B$  in $\tore$. Modify $\W$ by changing $\tore$ into $\tore\cup\A\setminus int(\B)$ and $\tore_j$ into $\tore_j\cup\s\B\setminus int(\s \A)$, and perform a small isotopy so that the number of components in $\W\cap\s\W$ decreases (cf. figure \ref{deform3}).  Pursue this process until
$\W\cap \s\W$ has no more closed curve component. 

\begin{figure}[h]
\centerline{\includegraphics[scale=0.8]{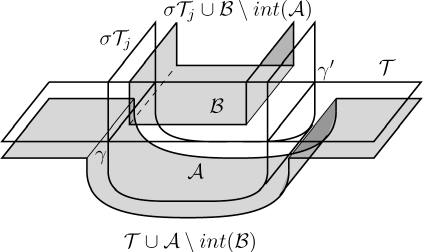}}
\caption{\small By considering small enough collar neighborhoods $N(\tore)$ of $\tore$ and $N(\s\tore_j)$ of $\s\tore_j$ in a subdivision of the triangulation of $\N$, one can deform by isotopy the tori $\tore\cup\A\setminus int(\B)$ and $\s\tore\cup\B\setminus int(\A)$ so that their number of intersection decreases.}
\label{deform3}
\end{figure}

Up to this stage  $\W$ is a JSJ decomposition of $\widehat{\N}$; its components
are 2-sided essential tori that 
  fall in two parts: those with $\s\tore\subset\W$ and those with $\s\tore\cap\W=\varnothing$.
For those $\tore$ such that $\s\tore\cap\tore=\varnothing$, $p(\tore)$ is a two-sided essential torus in $\M$.
 For those $\tore$ such that $\s\tore=\tore$ and $\s$ is orientation reversing on $\tore$, $p(\tore)$ is a two-sided Klein bottle in $\M$. For those $\tore$ such that $\s\tore=\tore$ and $\s$ is orientation preserving on $\tore$, $p(\tore)$ is a one-sided torus.
In the latter case consider a regular neighborhood $V(\tore)$ of $\tore$ in $\N$ with $\s V(\tore)=V(\tore)$  and change  $\tore$ in $\W$ by a component of $\partial V(\tore)$.

Finally for any $\tore$ such that $\s\tore\cap\W=\varnothing$: by the characteristic pair Theorem (cf. \cite {js}), $\s\tore$ is parallel 
to some $\tore'\subset\W$. If $\tore'=\tore$, then $\tore$ and $\s\tore$ cobound a $\TT^2\times\I$ preserved under $\s$ or $\N$ is a torus bundle modeled on $Sol$ geometry (cf. Theorem 5.3, \cite{scott}). If $\tore'\not=\tore$: note that $\s\tore'=\tore$ and replace in $\W$ its component $\tore'$ by $\s\tore$. 
At the end of the process $\W$ and $\Xi=p(\W)$ are coherent JSJ decompositions of $\N$ and $\M$.
\end{proof}

{\noindent\sl$\bullet$\ \begin{minipage}[t]{\largeur}
{Apply  Algorithm $\mathfrak{Top3}$ to find coherent JSJ decompositions $\W_i$ of the $\N_i$'s and $\Xi_i$ of the $\M_i$'s.}
\medskip
\end{minipage}
}

Define the surfaces $\W=\bigcup_i \W_i$ and $\Xi=\bigcup_i\Xi_i$ respectively embedded in $\N$, $\M$. Note that the involution $\s$ naturally acts on $\N\setminus\Sigma\sqcup\W\sqcup\s\W$, permuting its connected components, and consider the two $\s$-equivariant maps $p:\N\setminus (\Sigma\sqcup\W\sqcup\s\W)\longrightarrow \M\setminus(\Pi\sqcup\Xi)$ and $p:\Sigma\sqcup\W\sqcup\s\W\longrightarrow \Pi\sqcup\Xi$ obtained by restriction of $p:\N\longrightarrow \M$.

\begin{lem}[Algorithm $\mathfrak{Top4}$]\label{seifertinvariant} For each component $\piece$  of $\N\setminus\Sigma\sqcup\W\sqcup\s\W$, let $\widehat{Q}$ be obtained by gluing balls to all $\SS^2\subset\partial \piece$.
There is an algorithm which checks for each such $\piece$ whether $\widehat{\piece}$ is a Seifert fiber space and if so returns a set of Seifert invariants. 
\end{lem}

\begin{proof}
The algorithm is given in  \cite[Algorithm 8.1]{jt} and implicitly provides a set of Seifert invariants.\smallskip
\end{proof}

{\noindent\sl$\bullet$\ \begin{minipage}[t]{\largeur}
{
Apply Algorithm $\mathfrak{Top4}$ to decide which of the pieces of $\N\setminus\Sigma\sqcup\W\sqcup\s\W$ are punctured Seifert fiber spaces and  return for each a set of Seifert invariants.
}
\end{minipage}
}


\subsection{\bf Step 3: Graph of groups splittings of ${\P(\M)}$ and
${\P(\N)}$.} We now focus on how $\pi_1(\M)$ and $\pi_1(\N)$ can be given constructively by finite sets of data. This can be achieved by constructing graphs of group related to the topological decompositions of $\M$ and $\N$ with vertex and edge groups given by finite presentations.\medskip

First we need to establish  the following  algorithms that will be useful in the remaining of this part. We say that a 3-manifold $\V$ has {\sl incompressible boundary} if  for any component $\tore$ of $\partial \V$, the inclusion map $i:\tore\hookrightarrow \V$ induces a monomorphism $i_*:\pi_1(\tore)\lra \pi_1(\V)$. For a 3-manifold with incompressible boundary a {\sl peripheral subgroups system} is a collection of monomorphisms from the $\pi_1$ of components of $\partial\V$ into $\pi_1(\V)$ induced by the inclusion maps; each such monomorphism  is only well defined up to conjugacy in $\pi_1(\V)$.


\begin{lem}[Basic algorithms in the $\pi_1$ of the pieces] \label{algopieces} Let $\V$ be  a 3-manifold given by a triangulation,  $q:\WW\lra\V$ be the orientation cover,   $V=\pi_1(\V)$ and $W=\pi_1(\WW)$ seen as a subgroup of $V$ (whenever $\V$ is orientable $q:\WW\lra\V$ is an homeomorphism).

\begin{itemize}
\item[(i)] {\rm(Finite presentations).} One can algorithmically produce  finite presentations $<S|R>$ of $V$ and $<S'|R'>$ of $W$ with $S'$  a set of words on  $S\cup S^{-1}$.\smallskip
\item[(ii)] {\rm(Algorithm $\mathfrak{Gwp}(W,V)$).} Given a word $w$ on  $S\cup S^{-1}$ one can decide whether $w\in W$ and if so produce a word $w'$ on $S'\cup S'^{-1}$ which represents the same element.
\end{itemize}
In the following $\partial\V$ is  incompressible and consist of $\SS^2$, $\PP^2$, $\TT^2$, $\KB^2$. 
\begin{itemize}
\item[(iii)] {\rm (Peripheral subgroups system).} 
 One can construct a peripheral subgroups system $(V_i)_{i=1\ldots p}$ of $V$ (respectively $(W_j)_{j=1\ldots q}$ of $W$)  by canonical finite presentations with  generators words on $S\cup S^{-1}$ (respectively on $S'\cup S'^{-1}$),  represented by loops in $\partial \V$ (respectively $\partial\WW$) and such that for any $i=1\ldots p$, $V_i\cap W=W_i$.\smallskip
\item[(iv)] {\rm (Generalized word problem for boundary subgroups).} Here we moreover suppose that $\V$ is geometrisable.
Given peripheral subgroups systems $(V_i)_i$ of $V$ and $(W_j)_j$ of $W$, and $w$ a word on $S\cup S^{-1}$ (respectively on $S'\cup S'^{-1}$), one can decide whether $w\in V_i$ (respectively $w\in W_j$) and if so find a word $w'$ on generators of $V_i$ (respectively of $W_j$) which represents the same element in $V$ (respectively in $W$).
\end{itemize}
\end{lem}

\begin{proof} In case $\V$ is orientable, $q:\WW\lra\V$ is an homeomorphism, and simply skip in the lines of the proof all assumptions involving $\WW$ or $W$.\smallskip

\noindent {\it Proof of {\rm (i)} and {\rm (ii)}}.
As in the  proof of Lemma \ref{l4}, construct from a triangulation of $\V$ a PL-ball $\B$ with a PL-identification $f$ of triangles in $\partial\B$ with quotient manifold homeomorphic to $\V$.
Let $D_f$, the {\sl domain of $f$}, be the union of all triangles in $\partial \B$ identified by $f$.
Choose a point $*$ in $\B$. For any PL-triangle $\delta$ in $D_f$ apply barycentric subdivisions to $\delta$ and $f(\delta)$ and then  choose a PL-loop $\delta_1$ in $\B$ from $*$ to the center of gravity $\delta_*$ of $\delta$, and a PL-loop $\delta_2$ in $\B$ from $f(\delta_*)$ to $*$;  consider the PL-loop  $l(\delta)=\delta_1\delta_2$ based in $*$ (cf. figure \ref{figgenerator}); let $\lambda S=\{l(\delta)\,,\, \delta\in D_f\}$ be the finite set of all PL-loops based in $*$ obtained in this way. Given any PL-loop $l$ in $\V$ based in $*$ there is an  algorithm which homotopically changes, with $*$ fixed, the $*$-loop $l$ into a product of  elements of $\lambda S$: simply deform slightly $l$ so that it becomes transverse with $\partial\B$ then look at the successive triangles in $D_f$ that $l$ passes through,  to write it down as a product of $*$-loops in $\lambda S$.  In particular $\lambda S$ is a set of representatives of a generating set $S$ of $V\simeq \pi_1(\V,*)$ that we fix throughout the rest of the proof.
\begin{figure}[h]
\centerline{\includegraphics[scale=0.9]{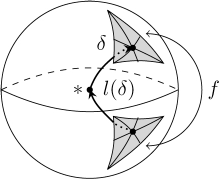}}
\caption{\small The loop $l(\delta)$ based in $*$ in $\V$ defined by a triangle $\delta$ with identification in $\partial \BB^3$; it yields a generator of $\pi_1(\V,*)$.}
\label{figgenerator}
\end{figure}

 Use the algorithm in \cite{pi_1} to compute a  finite presentation $<T|U>$ of $\pi_1(\V,*)$ from the triangulation of $\V$.
It  considers its 1-skeleton ${\rm K}$ and constructs a spanning tree ${\rm T}$ of ${\rm K}$;  for any edge $\e$ in ${\rm K}\setminus {\rm T}$ let  $l_1$ be the simple PL-loop in ${\rm T}$  from $*$ to the origin of $\e$ and $l_2$ the  simple PL-loop in ${\rm T}$ from the extremity of $\e$ to $*$ and let $l_\e=l_1\e\,l_2$  a $*$-loop which passes through $\e$. The set $T$ of generators is represented by the set of all $*$-loops $l_\e$ obtained in this way. For any PL-loop $l$ based in $*$ one    algorithmically constructs a  product of the $l_\e$, $\e\in {\rm K}\setminus {\rm T}$, homotopic to $l$ with $*$ fixed, by reading  the successive edges of ${\rm K}\setminus {\rm T}$ which appears in $l$.

Use the two processes described above to write down elements  $s\in S$ as words $T(s)$ on $T\cup T^{-1}$, and elements $t\in T$ as words $S(t)$ on $S\cup S^{-1}$ and apply the following sequence of Tietze transformations (cf. \cite{mks}) to the presentation $<T|U>$:
\begin{itemize}
\item[--] add a generator $s$ and a relation $s=T(s)$ for each $s\in S$, to  obtain $<S\cup T|U\cup U_1>$
\item[--] add relations $t=S(t)$ for all $t\in T$, to obtain $<S\cup T| U\cup U_1\cup U_2>$,
\item[--] use relations in $U_2$ to change each relation in $U\cup U_1$ and express  it on the alphabet $S$, to obtain $<S\cup T|U'\cup U'_1\cup U_2>$, 
\item[--] delete generators in $T$ and relations in $U_2$, to obtain $<S|U'\cup U'_1>$,
\end{itemize}
which finally yields a finite presentation of $V$ with generating set $S$ and proves the first assumption in (i). This presentation has a large number of generators and can be easily improved by identifying generators making use of a splitting of $D_f$ into connected surfaces on which $f$ provides  homeomorphisms.\smallskip

Among the set $S$ of generators one can decide which one reverses orientation and which one doesn't: indeed the loop $l(\delta)$ is orientation reversing if and only if $f_{|\delta}:\delta\lra f(\delta)$  reverses the orientation induced on $\delta,f(\delta)$ by that of $\B$.  Hence given a word on $S\cup S^{-1}$ one can decide if it represents an element in $W$ simply by counting whether it has an even occurrence of orientation reversing generators or not.  Consider the set $S'$ of words of one of the form: $s$, or $s's''$ or $s'ss'^{-1}$ for any $s$  orientation preserving, and any $s',s''$ orientation reversing, elements of $S\cup S^{-1}$. Each word on $S\cup S^{-1}$ having an even number of occurrence  of orientation reversing element can be easily written (in linear time) as a word on $S'\cup S'^{-1}$ for example by the deterministic pushdown automata in figure \ref{figautomata}. This shows that $S'$ generates $W$ and proves (ii).
\begin{figure}[h]
\centerline{\includegraphics[scale=0.8]{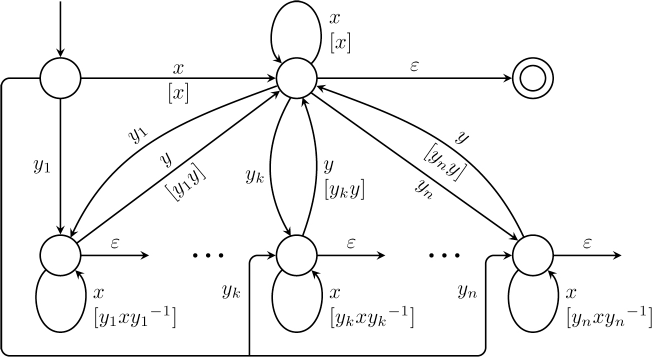}}
\caption{\small A pushdown automata that given a word on generators $S\cup S^{-1}$ of $V$ decides whether it represents an element in the index 2 subgroup $W$ and if so  returns in the stack a representative as a word on the generators $S'\cup S'^{-1}$  of $W$ as described above. Elements of $S\cup S^{-1}$ are denoted $x$ if they lie in  $W$ and $y,y_1,\ldots, y_n$ otherwise;  $\varepsilon$ is the empty string; in bracket is the element of $S'\cup S'^{-1}$ pushed on the top of the stack.}
\label{figautomata}
\end{figure}

Finally apply a process such as Reidemeister-Schreier (cf. \cite{mks,johnson}) to build a finite presentation $<T'|U'>$ of $W$, then express each generator in $S'$ as a word on $T'\cup T'^{-1}$ and each generator in $T'$ as a word on $S'\cup S'^{-1}$ and apply Tietze transformations as above to obtain a finite presentation $<S'|R'>$ with generators $S'$ for $W$. This completes the proof of (i).\smallskip\\
\noindent {\it Proof of {\rm (iii)}}.  First note that since $q:\WW\longrightarrow \V$ induces a monomorphism on fundamental groups, whenever $\V$ has  incompressible boundary, so has $\WW$. The set of PL-triangles in $\partial\B\setminus D_f$ together with $f$ provides a triangulation of $\partial\M$. Use it to compute the Euler characteristic $\chi$ and check orientability for each component of $\partial\M$; that determines their homeomorphism classes $\SS^2$, $\PP^2$, $\TT^2$ or $\KB^2$ depending on whether $\chi=2$, $\chi=1$,   $\chi=0$ and orientable, or  $\chi=0$ and non-orientable. Then  by a favorite trick (such as representing each surface in $\partial \M$ by a PL-disk with identification on its boundary edges and deforming to get one of 4 standard models, cf. \cite{seifth}) find for each  surface $\not\approx\SS^2$  in $\partial\M$ a family of $2-\chi$  PL-curves which represent generators of one of the canonical presentations $<a|a^2=1>$, $<a,b|[a,b]=1>$, or $<a,b|aba^{-1}=b^{-1}>$ of $\Z_2$, $\ZZ$, or $\Z\rtimes\Z$. Finally, use the algorithm described in the proof of (i) above to write down the generators on the alphabet $S\cup S^{-1}$. It defines a peripheral subgroups system of $V$. \smallskip

An element in a peripheral subgroup  of $V$ can be orientation preserving/reversing as an element of $V$, but also as an element of the surface group; note that since surfaces arise from  the boundary those two notions coincide. 
The boundary of $\WW$ is $q^{-1}(\partial\V)$ and its components are related to those of $\partial\V$:\smallskip\\
-- any $\SS^2,\,\PP^2\hookrightarrow \partial\V$ lifts to $\SS^2\hookrightarrow\partial \WW$. Since $\pi_1(\SS^2)=\{1\}$, the monomorphism $\pi_1(\SS^2)\lra\pi_1(\WW)$ is well defined.\smallskip\\
-- any $\TT^2\hookrightarrow \partial\V$ lifts to two components $\TT^2\hookrightarrow\partial\WW$. The monomorphism $\phi:\pi_1(\TT^2)\longrightarrow V$ has its image in $W$; 
let $\a$, $\b$ be two based loops in $\TT^2$ which represent generators $[\a]$, $[\b]$ of $\pi_1(\TT^2)$, $i:\TT^2\hookrightarrow \V$, and $a=\phi([\a])$, $b=\phi([\b])$; the loops $i(\a)$, $i(\b)$ are orientation preserving in $\V$ and lift to loops $\a^+,\b^+$ and $\a^-$, $\b^-$ lying in the two  $\TT^2\hookrightarrow\partial\WW$ where they both represent  a basis of $\pi_1(\TT^2)$. One defines the respective monomorphisms $\phi_+:\pi_1(\TT^2)\lra W$ by $\phi_+([\a^+])=a$, $\phi_+([\b^+])=b$ and $\phi_-:\pi_1(\TT^2)\lra W$, by
$\phi_-([\a^-])=vav^{-1}$, $\phi_-([\b^-])=vbv^{-1}$
 for some arbitrary element $v$ of $V\setminus W$.\smallskip\\
-- any $\KB^2\hookrightarrow\partial\V$ lifts to a $\TT^2\hookrightarrow\partial\WW$. Let  $\pi_1(\KB^2)\lra V$ be the monomorphism found above, and let $\a$, $\b$ be based loops in $\KB^2$  such that $\phi([\a])=a$ and $\phi([\b])=b$ for some generators $a,b$ of $\Z\rtimes\Z$ as in the presentation above. Let $i:\KB^2\hookrightarrow \V$; $i(\a)$, $i(\b)$ are respectively orientation reversing and orientation preserving loops in $\V$, then consider the two loops $\a_2$ and $\b_1$, respective lifts  of $i(\a)^2$ and of $i(\b)$ in $\partial\WW$; they represent generators of $\pi_1(\TT^2)$. Let the monomorphism $\phi':\pi_1(\TT^2)\lra W$ be defined by $\phi'([\a_2])=a^2$ and $\phi'([\b_1])=b$.\smallskip\\
We have finally constructed a peripheral subgroups system in $W$, which proves (iii). By construction whenever $V_i$ is a peripheral subgroup of $V$, $V_i\cap W$ is a peripheral subgroup $W_i$ of $W$.\smallskip\\
\noindent
{\it Proof of {\rm (iv)}}. Since $\V$ is geometrisable the word problem is solvable in $V$ (cf. \cite{epstein}). In particular in case of the peripheral subgroup $\{1\}$ coming from a $\SS^2$ component one can solve the generalized word problem.

Suppose first that $\V$ is orientable. Then any peripheral subgroup $V_1\not=\{1\}$ comes from a torus $\TT^2\hookrightarrow \partial\V$. Consider an homeomorphic copy $\V'$ of $\V$, and the double $2\V=\V\sqcup_{\TT^2} \V'$ of $\V$  along the boundary component $\TT^2\hookrightarrow \partial \V$.
Then  \cite[Lemma 1.2]{cp3mg} shows that $2\V$ is geometrisable, therefore $\pi_1(2\V)$ has solvable word problem. The group $\pi_1(2\V)$ splits into the amalgam $V*_{V_1} V'$ equipped with the isomorphism $v\in V\longmapsto v'\in V'$ which restricts to the identity on the subgroup $V_1\simeq\ZZ$. Let $v$ be an element of $V$ given as a word on $S\cup S^{-1}$, then by the normal form theorem for amalgams (cf. \cite{mks}), $v\in V_1$ if and only if $v^{-1}v'=1$ in $\pi_1(2\V)$, which can be checked using the solution to the word problem in $\pi_1(2\V)$.  If yes one can enumerate elements in $V_1$ as words on the generators found in (iii) and for each use a solution to the word problem in $V$ to decide whether it equals $v$, to finally write down $v$ as a word on generators of $V_1$;  this naive process can be improved using the quasi-convexity of $V_1$ in $V$ (cf. \cite{epstein}). (Note that in case $\V$ is an orientable piece coming from a JSJ decomposition of a closed irreducible 3-manifold --as occurs in our context-- a far more efficient solution is given by \cite[Proposition 4.2]{cp3mg}). In particular this solves the generalized word problem for peripheral subgroups of $W$.

Suppose now that $\V$ is non-orientable.
Let $V_1\not=\{1\}$ be a peripheral subgroup of $V$; $V_1$
is isomorphic either to $\Z_2$, $\ZZ$ or $\Z\rtimes\Z$.
 If $V_1\simeq \Z_2$  is generated by $a$ then $v$ lies in
$V_1$ if and only if $v$ commutes with $a$ (cf. \cite{swarup}) and one can decide using a solution to the word problem, and if so write $v$ as a word, $1$ or $a$,  on the generators of $V_1$ . If
$V_1\simeq\ZZ$, then $V_1\subset W$ is also, by (iii), a peripheral subgroup of $W$. One checks first using (ii) whether $v\in W$, and if so, the,  solves the problem  reduced to $W$  by the solution given above for the oriented case.
 If
$V_1\simeq \Z\rtimes\Z$; let $t\in V_1\setminus (V_1\cap W)$ be an o.r. element in $V_1$ provided by (iii).
 One  decides  whether $v$ lies in $W$ other not; then using the above solution for the oriented case,  one decides  whether,  $v$ in the former case, or $vt$ in the latter case, lies in the peripheral subgroups $V_1\cap W$ of $W$. If yes it provides a word on the generators of $V_1$ equal to $v$. In any case this solves the generalized word problem in  a peripheral subgroup of $V$ and proves (iv).\smallskip
\end{proof}


We now turn to the description of $\pi_1(\M)$ and $\pi_1(\N)$ using graphs of groups related to the topological decompositions obtained in  step 2.\smallskip

 Recall that a {\sl graph of group} $\X$ consists of (cf. \cite{serre}):
\begin{itemize}
\item[--] a non-empty finite connected oriented graph ${\rm X}$; let ${\rm VX}$, ${\rm EX}$ denote respectively the vertex and edge sets of ${\rm X}$,
for all $\e\in {\rm EX}$, $\bar{\e}$ denotes the opposite edge of $\e$, and $\t(\e)\in {\rm VX}$ denotes the extremity of $\e$; the edge
$\e$ has origin $\t(\bar{\e})$ and extremity $\t(\e)$,\smallskip 
\item[--] two families of {\sl vertex groups}  $G(\v)$ for all ${\v\in {\rm VX}}$ and {\sl edge groups} $G(\e)$ for all ${\e\in {\rm EX}}$, with
$G({\bar{\e}})=G(\e)$,\smallskip
\item[--] a family of monomorphisms $\f_\e:G(\e)\longrightarrow G({\t(\e)})$ for all ${\e\in {\rm EX}}$.\smallskip
\end{itemize}

Graphs of groups come equipped with the notion of {\sl fundamental group} of  graph of group $\X$ (cf. \cite{serre,bass}), that we introduce now.
An $\X${\sl -path} of {\sl length} $n\in\Bbb N$ is a finite sequence $(g_0,\e_1,g_1,\ldots,$ $\e_n,g_n)$ such that $\forall\,i=1 \ldots n-1$, one has $\t(\e_{i})=\t(\ol{\e}_{i+1})$, $g_i\in G(\t(\e_i))$, $g_0\in G(\t(\ol{\e}_1))$ and $g_n\in G(\t(\e_n))$.  We denote by $\pi(\X)$ the set of $\X$-paths.  
 An $\X$-path is {\sl reduced} if it does not contain  a subsequence $(\ldots,\e,\f_\e(g),\ol{\e},\ldots)$ for some $\e\in {\rm EX}$, $g\in G(\e)$; any $\X$-path can be transformed into 
a  reduced $\X$-path, called a reduction,  by changing each subsequence of the above form using relations:
\begin{equation}\tag{$*$}
(\ldots,g',\e,\f_\e(g),\ol{\e},g'',\ldots) \equiv (\ldots,g'\f_{\ol{\e}}(g)g'',\ldots)
\end{equation}
 for any $\e\in {\rm EX}$ and $g\in G(\e)$; moreover any two reductions of an $\X$-path must have the same length. 
The set $\pi(\X)$  comes equipped with a partially defined  {\sl concatenation product}:  $(\ldots,\e_1,g_1)(g_{2},\e_{2},\ldots)=(\ldots,\e_1,g_1g_2,\e_2,\ldots)$ anytime $\t(\e_1)=\t(\ol{\e}_2)$.
Let $\x\in {\rm VX}$; a {\sl $(\X,\x)$-loop} is a $\X$-path such that $\t(\ol{\e}_1)=\t(\e_n)=\x$. The concatenation product is well defined on $(\X,\x)$-loops and is compatible with relations ($*$); the equivalence classes of $(\X,\x)$-loops with respect to relation ($*$) inherits  a group structure, and we denote this group by  $\pi_1(\X,\x)$, the {\sl fundamental group}  of $\X$ based in $\x$. Each element of $\pi_1(\X,\x)$ can be represented by a reduced $(\X,\x)$-loop, which allows to define its length. The isomorphism class
of $\pi_1(\X,\x)$ does not depend on the base point $\x\in {\rm VX}$ and will be denoted by $\pi_1(\X)$.\\

Given a 3-manifold $\V$ and a two-sided compact incompressible surface $\Phi$ in $\V$ there is a usual way to define a {\sl graph of group $\mathbf{V}$ related to $(\V,\Phi)$} with $\pi_1(\mathbf{V})\simeq\pi_1(\V)$. Consider the interior $N(\Phi)$  of a regular neighborhood of $\Phi$ in $\V$. The vertices $\v_i$ (respectively the edges $\e_j$) of $\mathbf{V}$ are in 1-1 correspondence with the components $\V_i$ of $\V\setminus N(\Phi)$ (respectively with the components $\tore_j$ of $\Phi$), vertex groups (respectively edge groups) are $G(v_i)=\pi_1(\V_i)$ (respectively $G(\e_j)=\pi_1(\tore_j)$). The embedding of $\Phi$ in $\V$ defines for each $\tore_j\in\Phi$ two embeddings $f_j^+,f_j^-$ of $\tore_j$ into the boundary of some components $\V_i$, $\V_k$ of $\V\setminus N(\Phi)$ (possibly $i=k$) that induce two monomorphisms $g_j^+,g_j^-$ of their $\pi_1$. With the identifications above one defines $\phi_{\e_j}=g_j^+:G(\e_j)\lra G(\t(\e_j))$ and $\phi_{\ol{\e}_j}=g_j^-:G(\ol{\e}_j)\lra G(\t(\ol{\e}_j))$. This defines a graph of group $\mathbf{V}$ which depends on all the monomorphisms $g_j^+,g_j^-$ despite the isomorphism class of its fundamental group $\pi_1(\mathbf{V})$ does not. One proves by applying the Seifert-Van Kampen Theorem that $\pi_1(\mathbf{V})\simeq\pi_1(\V)$.\\

Given topological decompositions $\Sigma\sqcup\W$ of $\N$ and $\Pi\sqcup\Xi$ of $\M$ we consider a graph of group $\Mg$ related to $(\M,\Pi\sqcup\Xi)$ and a graph of group $\Ng$ related to $(\N,\Sigma\sqcup\W\sqcup\s\W)$ (rather than on  $(\N,\Sigma\sqcup\W)$). This last graph of group slightly differs from that related to $(\N,\Sigma\sqcup\W)$ in that it may be non-minimal ({\sl i.e.} it may contain an edge $\e$ with $\t(\e)\not=\t(\ol{\e})$ and $\vf_\e$ is onto); it's the prize for having a {\sl covering} of graphs of groups. More precisely, by {\sl coherent graph of group} decompositions for $\pi_1(\M)$ and $\pi_1(\N)$ we mean:\smallskip\\
-- a graph of groups $\Mg$ related to $(\M,\Pi\sqcup\Xi)$; $\pi_1(\Mg)\simeq\pi_1(\M)$,\smallskip\\
-- a graph of group $\Ng$ related to $(\N,\Sigma\cup\W\cup\s\W)$; $\pi_1(\Ng)\simeq\pi_1(\N)$,\smallskip\\
-- a {\sl covering} $\mathbf{p}:\Ng\longrightarrow \Mg$, that is a collection of: \label{coverpage}
\begin{itemize}
\item[--] a map of graphs  $\mathrm{p}:{\rm N}\longrightarrow {\rm M}$ from the underlying graph ${\rm N}$ of $\Ng$ to the underlying graph ${\rm M}$ of $\Mg$, induced by $p:\N\lra\M$, 
\item[--] two families of monomorphisms $\p_\v:G(\v)\lra G(\p(\v))\,,\,\v\in {\rm VN}$ and $\p_\e:G(\e)\lra G(\p(\e))\,,\,\e\in {\rm EN}$,\smallskip
\item[--] a collection of elements $\mu(\e)$, $\e\in {\rm EN}$, with  $\mu(\e)\in G(\t(\p(\e)))$ such that if $ad_\e$ is the automorphism of $G(\t(\p(\e)))$ defined by $\forall\,g\in G(\t(\p(\e))),\, ad_\e(g)=\mu(\e)\,g\,\mu(\e)^{-1}$, the following diagram commutes:
\[
\begin{CD}
G(\e) @>\phi_\e>> G(\t(\e))\\
@V\p_\e VV @VV\p_{\t(\e)}V\\
G(\p(\e)) @>>ad_\e\circ\,\phi_{\p(\e)}>  G(\t(\p(\e)))
\end{CD}
\]
\item[--] a map $\p_\#:\pi(\Ng)\lra \pi(\Mg)$ defined by:
$$
\p_\#(g_0,\e_1,g_1,\e_2,\ldots ,\e_n,g_n)=(g'_0,\p(\e_1),g'_1,\p(\e_2),\ldots,\p(\e_n),g'_n)\\
$$
where:
$$\forall\,i=0,\ldots,n,\ g'_i=\left\lbrace
\begin{array}{ll}
\p_{\t(\ol{\e}_1)}(g_0)\,\mu(\ol{\e}_1)^{-1}& \text{for}\ i=0 \\
\mu(\e_i)\,\p_{\t(\e_i)}(g_i)\,\mu(\ol{\e}_{i+1})^{-1}& \text{for}\  i\not=0,n\\
\mu(\e_n)\;\p_{\t(\e_n)}(g_n)& \text{for}\ i=n
\end{array}\right.
$$
and for any $\x\in {\rm VM}$, $\tl{\x}\in \p^{-1}(\x)$, $\p_\#$ induces a monomorphism:
$$
\p_*:\pi_1(\Ng,\tl{\x})\lra\pi_1(\Mg,\x)~.
$$
\end{itemize}
One may refer to \cite{bass} for a general definition of a covering of graphs of groups; its formalism differs from our, which turns to be more practical in the present context though less general; we won't need to relate to the definition of \cite{bass} in our purpose.

\begin{lem}[Algorithm  for graphs of groups]\label{graphsofgroups}
One can algorithmically produce coherent graphs of groups decomposition $\Ng$ and $\Mg$ for $\pi_1(\N)$ and $\pi_1(\M)$ related to the topological decompositions $(\N,\Sigma\sqcup\W\sqcup\s\W)$ and $(\M,\Pi\sqcup\Xi)$ as well as  a covering of graphs of groups $\pp:\Ng\lra\Mg$ and the induced monomorphism $\p_*:\pi_1(\Ng,\tl{\x})\longrightarrow \pi_1(\Mg,\x)$ (given any vertex $\x$ of the underlying graph $\rm M$, and $\tl{\x}\in\p^{-1}(\x)$).
\end{lem}

\begin{proof}
 The graph of group $\Mg$ is deduced from the topological decomposition of $\M$ along $\Pi\sqcup\Xi$ obtained by algorithms $\mathfrak{Top2}$, $\mathfrak{Top3}$ and from finite presentations of the fundamental groups of  the pieces  obtained in Lemma \ref{algopieces}.(i) together with their peripheral subgroups systems given algorithmically by Lemma \ref{algopieces}.(iii).

If ${\rm M}$ denotes the underlying graph of $\Mg$, ${\rm VM}$ is in 1-1 correspondence with the connected components of $\M\setminus (\Pi\sqcup\Xi)$ and ${\rm EM}$  is in 1-1 correspondence with the components of $\Pi\cup\Xi$. For each $\v\in {\rm VM}$, $G(\v)$ is the fundamental  group of the corresponding component of $\M\setminus (\Pi\sqcup\Xi)$, and for each
$\e\in {\rm EM}$, $G(\e)=\Z_2,\,\ZZ,$ or $\pK$  according to the associated component is homeomorphic to  $\PP^2,\TT^2$ or $\KB^2$. The monomorphisms $\phi_\e:G(\e)\lra G(\t(\e))$ are induced by the sewing maps together with the peripheral subgroups systems in all vertex groups.

Now that a graph of group $\Mg$ with $\pi_1(\Mg)\simeq \pi_1(\M)$ associated to the splitting of $\M$ along $\Pi\sqcup\Xi$  is given   we construct from $\Mg$  a related graph of group splitting $\Ng$ of $\pi_1(\N)$. This graph of group $\Ng$ is related to the topological decomposition of $\N$ along $\Sigma\sqcup \W\sqcup\s\W$: despite we only focus on the graph of group, keep in mind in the line of the proof that the construction of $\Ng$ encodes how $\N$ and  $\Sigma \sqcup \W\sqcup\s\W$ are constructed by gluing the orientation coverings of components of $\M\setminus(\Pi\sqcup \Xi)$.\smallskip 

Partition ${\rm VM}$ into 
$${\rm VM}={\rm VM}^+\sqcup {\rm VM}^-,\quad {\rm VM}^+=\left\{ \v_1,\ldots,\v_q\right\rbrace,\quad {\rm VM}^-=\{\v_{q+1},\ldots,\v_r\}$$ 
where ${\rm VM}^+$ are those vertices coming from oriented components and ${\rm VM}^-$  those coming from non-orientable components of $\M\setminus (\Pi\sqcup\Xi)$ (orientability of the pieces in $\M\setminus\Pi\sqcup\Xi$  can be algorithmically checked from their triangulations).
 For any $\v\in {\rm VM}^-$, $G(\v)\simeq \pi_1(\M_i)$ naturally arises with its index two subgroup  of orientation preserving elements described by Lemma \ref{algopieces}.(i), that we denote $H(\v^+)$; choose for any $\v\in {\rm VM}^-$ an arbitrary element $\mu(\v)$ in $G(\v)\smallsetminus H(\v^+)$, which defines an automorphism $ad_\v$ of $H(\v^+)$ by 
$\forall\,h\in H(\v^+),\,ad_\v(h)=\mu(\v)\,h\,\mu(\v)^{-1}$. Moreover the choice of $\mu(\v)$ defines a peripheral subgroups system of $H(\v^+)$, as observed in the proof of Lemma \ref{algopieces}.(iii).

Similarly partition ${\rm EM}$ into 
$${\rm EM}={\rm EM}^+\sqcup {\rm EM}^-,\quad {\rm EM}^+=\{\e_1,\ldots,\e_s\}, \quad {\rm EM}^-=\{\e_{s+1},\ldots,\e_t\}$$ 
where ${\rm EM}^+$ are edges associated to  $\TT^2$ and ${\rm EM}^-$ edges associated to $\PP^2$ or $\KB^2$; note that $\t({\rm EM}^-)\subset {\rm VM}^-$ and that $\e\in {\rm EM}^-$ if and only if $\ol{\e}\in {\rm EM}^-$.

 One constructs the graph ${\rm N}$ by picking $q+r$ vertices and $s+t$ edges:
\begin{gather*}
{\rm VN}=\{\v^+_1,\ldots,\v^+_q,\v^-_1,\ldots,\v^-_q,\v^+_{q+1},\ldots,\v^+_r\},\\
 {\rm EN}=\{\e^+_1,\ldots,\e^+_s,\e^-_1,\ldots ,\e^-_s,\e^+_{s+1},\ldots,\e^+_t\}
\end{gather*}
 and by setting:
$$
\begin{array}{ll}
\forall\,i=1 \ldots  \t,\quad \t(\e^+_i) &=\t(\e_i)^+\\
\forall\,i=1 \ldots s,\quad \t(\e^-_i) &=
\begin{cases}
\t(\e_i)^+ &\text{whenever}\ \t(\e_i)\in {\rm VM}^-\\
\t(\e_i)^- &\text{whenever}\ \t(\e_i)\in {\rm VM}^+.
\end{cases}
\end{array}
$$
Define the map of graphs $\p:{\rm N}\lra {\rm M}$ by $\p(\v^{\pm})=\v$ and $\p(\e^\pm)=\e$.

The vertex groups $(H(\v))_{\v\in {\rm VN}}$ of $\Ng$ are defined by:

$$
\begin{array}{ll}
\forall\,\v\in {\rm VM}^+,\quad &H(\v^+)=H(\v^-)=G(\v)\\
\forall\,\v\in {\rm VM}^-,\quad &H(\v^+)\lhd_2 G(\v)
\end{array} 
$$
where $H(\v^+)\lhd_2 G(\v)$ is the subgroup of orientation preserving elements as discussed above. The edge subgroups $(H(\e))_{\e\in {\rm EN}}$ of $\Ng$ are defined by:
$$
\begin{array}{ll}
\forall\,\e\in {\rm EM}^+, &H(\e^+)=H(\e^-)=G(\e)=\ZZ\\
\forall\,\e\in {\rm EM}^-, &H(\e^+) \lhd_2 G(\e)\ \text{and}\ H(\e^+)\simeq
\begin{cases}
\{1\} &\text{if}\ G(\e)=\Z_2\\
\ZZ &\text{if}\ G(\e)=\Z\rtimes\Z
\end{cases}
\end{array} 
$$

The monomorphisms $\phi_\e:H(\e)\lra H(\t(\e))$ for all $\e\in {\rm EN}$ of $\Ng$ are defined by:
\begin{itemize}
\item[(i)] Let $\e\in {\rm EM}^-$; necessarily $\t(\e)\in {\rm VM}^-$. The monomorphism $\phi_\e:G(\e)\longrightarrow G(\t(\e))$ sends the index 2 subgroup $H(\e^+)$ of $G(\e)$ into the index 2 subgroup $H(\t(\e)^+)$ of $G(\t(\e))$. Define $\phi_{\e^+}$ by the commutative diagram:
$$
\xymatrix{
H(\e^+) \ar@{^{(}->}^{2}[d] \ar[r]^{\phi_{\e^+}} & H(\t(\e)^+) \ar@{^{(}->}^{2}[d] \\
G(\e) \ar[r]_{\phi_\e} & G(\t(\e)) 
}
$$
\item[(ii)] Let $\e\in {\rm EM}^+$; here $H(\e^+)=H(\e^-)=G(\e)=\ZZ$. There are two cases:\smallskip
\begin{itemize}
\item[(ii.a)] If $\t(\e)\in {\rm VM}^+$; Define $\phi_{\e^+}$, $\phi_{\e^-}$ by the commutative diagrams:
$$
\xymatrix{
H(\e^+) \ar@{=}[d] \ar[r]^{\phi_{\e^+}} & H(\t(\e)^+) \ar@{=}[d] \\
G(\e) \ar[r]_{\phi_\e} & G(\t(\e)) 
}
\qquad
\xymatrix{
H(\e^-) \ar@{=}[d] \ar[r]^{\phi_{\e^-}} & H(\t(\e)^-) \ar@{=}[d] \\
G(\e) \ar[r]_{\phi_\e} & G(\t(\e)) 
}
$$
\item[(ii.b)] If $\t(\e)\in {\rm VM}^-$; in that case  $H(\t(\e)^+)\,\lhd_2\,G(\t(\e))$ and one has the automorphism $ad_{\t(\e)}$ of $H(\t(\e)^+)$ defined above. 
 Since $\e\in {\rm EM}^+$ and $\t(\e)\in {\rm VM}^-$,  one has $\phi_\e(G(\e))\subset H(\t(\e)^+)$. Define $\phi_{\e^+}$, $\phi_{\e^-}$ by the commutative diagrams:
$$
\xymatrix{
H(\e^+) \ar@{=}[d] \ar[r]^{\phi_{\e^+}} & H(\t(\e)^+) \ar@{^{(}->}^{2}[d] \\
G(\e) \ar[r]_{\phi_\e} & G(\t(\e)) 
}
\qquad
\xymatrix{
H(\e^-) \ar@{=}[d] \ar[r]^{\phi_{\e^-}} & H(\t(\e)^+) \ar@{^{(}->}^{2}[d] \\
G(\e) \ar[r]_{ad_{\t(\e)}\circ\phi_\e} & G(\t(\e)) 
}
$$
\end{itemize}
\end{itemize}

The graph ${\rm N}$ together with the families of vertex groups $H(\v)$, $\v\in {\rm VN}$, edge groups $H(\e)$, $\e\in {\rm EN}$ and monomorphisms $\phi_\e:H(\e)\longrightarrow H(\t(\e))$, $\e\in {\rm EN}$, defines the graph of group $\Ng$. The vertex and edge groups of $\Ng$ are subgroups  of vertex and edge subgroups of $\Mg$, which gives rise to the  two families of monomorphisms: 
\begin{gather*}
\p_\v:H(\v)\lra G(\p(\v)),
\v\in {\rm VN}\\
\p_\e:H(\e)\lra G(\p(\e)), \e\in {\rm EN}.
\end{gather*}
For each $\e\in {\rm EN}$ define $\mu(\e)\in G(\t(\p(\e)))$ by:
$$
\begin{array}{ll}
\text{If}\ \e\in {\rm EM}^-, &\mu(\e^+)=1\\
\text{If}\ \e\in {\rm EM}^+, &\mu(\e^+)=1,\quad
\mu(\e^-)=
\begin{cases}
1 &\text{if}\ \t(\e)\in {\rm VM}^+\\
\mu(\t(\e)) & \text{if}\ \t(\e)\in {\rm VM}^- 
\end{cases}
\end{array}
$$
Let $ad_{\e}$ be the automorphism of $G(\t(\p(\e)))$: $ad_\e(h)=\mu(\e)\, h\, \mu(\e)^{-1}$. By construction the following diagram commutes for all $\e\in {\rm EN}$:
\[
\begin{CD}
H(\e) @>\phi_\e>> H(\t(\e))\\
@V\p_\e VV @VV\p_{\t(\e)}V\\
G(\p(\e)) @>>ad_\e\circ\,\phi_{\p(\e)}>  G(\t(\p(\e)))
\end{CD}
\]
Consider $\p_\#:\Ng\lra\Mg$ as in the definition of covers of graphs of groups (cf p.\pageref{coverpage}). Let $\x\in {\rm VM}$ and  $\tl{\x}=\x^+\in \p^{-1}(\x)$; it remains to prove that $\p_\#$ induces a monomorphism $\p_*:\pi_1(\Ng,\tl{\x})\lra \pi_1(\Mg,\x)$. First $\p_\#$ induces an homomorphism $\p_*:\pi_1(\Ng,\tl{\x})\lra \pi_1(\Mg,\x)$, since whenever $\t(\e_n)=\t(\ol{\e}_{n+1})$:
$$
\begin{array}{ll}
&\p_\#(g_0,\e_1,\ldots,\e_n,g_n)\,\p_\#(h_n,\e_{n+1},\ldots,\e_m,g_m)\\
=&(g_0',\p(\e_1),\ldots,\p(\e_n),\mu(\e_n)\p_{\t(\e_n)}(g_n)\p_{\t(\ol{\e}_{n+1})}(h_n)\mu(\ol{\e}_{n+1})^{-1},
\p(\e_{n+1}),\\
&\qquad\qquad\qquad\qquad\qquad\qquad\qquad\qquad\qquad\qquad\quad\ldots,\p(\e_m),g'_m)\\
=&(g_0',\p(\e_1),\ldots,\p(\e_n),\mu(\e_n)\p_{\t(\e_n)}(g_nh_n)\mu(\ol{\e}_{n+1})^{-1},
\p(\e_{n+1}),\\
& \qquad\qquad\qquad\qquad\qquad\qquad\qquad\qquad\qquad\qquad\quad\ldots,\p(\e_m),g'_m)\\
=&\p_\#(g_0,\e_1,\ldots,\e_n,g_nh_n,\e_{n+1},\ldots,\e_m,g_m)\\
\end{array}
$$
Secondly $\p_*$ is injective: we prove that the image of a reduced $(\Ng,\tl{\x})$-loop is a reduced $(\Mg,\x)$-loop. Clearly the image of a  $(\Ng,\tl{\x})$-loop is a $(\Mg,\x)$-loop. Let $\gamma=(g_0,\e_1,\ldots,\e_n,g_n,\e_{n+1},\ldots,\e_m,g_m)$ be a reduced $(\Ng,\tl{\x})$-loop. Suppose that  $\p_\#(\gamma)$ is not reduced, more precisely that $\p(\e_n)=\p(\ol{\e}_{n+1})$ and that $\mu(\e_n)\p_{\t(\e_n)}(g_n)\mu(\ol{\e}_{n+1})^{-1}$ lies in $\phi_{\p(\e_n)}(G(\p(\e_n)))$. There are several cases to consider:
\begin{itemize}
\item[(i)] If $\p(\e_n)\in {\rm EM}^-$; here $\e_n=\ol{\e}_{n+1}$, $\mu(\e_n)=\mu(\ol{\e}_{n+1})$ and $\p_{\t(\e_n)}(g_n)$ lies in
$\phi_{\p(\e_n)}(G(\p(\e_n)))$ if and only if $g_n\in \phi_{\e_n}(H(\e_n))$; since $\phi_{\e_n}(H(\e_n))=\phi_{\p(\e_n)}(G(\p(\e_n)))\cap \p_{\t(\e_n)}(H(\t(\e_n)))$. In that case $\gamma$ is non-reduced.\smallskip
\item[(ii)] If $\p(\e_n)\in {\rm EM}^+$; there are two cases to consider:\smallskip
\begin{itemize}
\item[(ii.a)] if $\p(\t(\e_n))\in {\rm VM}^+$; here $\e_n\not= \ol{\e}_{n+1}$ implies $\t(\e_n)\not=\t(\ol{\e}_{n+1})$, hence $\e_n=\ol{\e}_{n+1}$. Moreover $\mu(\e_n)=\mu(\ol{\e}_{n+1})=1$. As above $\gamma$ is non-reduced.\smallskip
\item[(ii.b)] If $\p(\t(\e_n))\in {\rm VM}^-$; there are four cases to consider:\smallskip
\begin{itemize}
\item[(ii.b.1)] if $\e_n=\ol{\e}_{n+1}=\p(\e_n)^+$; then $\mu(\e_n)=\mu(\ol{\e}_{n+1})=1$ and as above $\gamma$ is non-reduced.
\item[(ii.b.2)] If $\e_n=\ol{\e}_{n+1}=\p(\e_n)^-$; then $\mu(\e_n)=\mu(\ol{\e}_{n+1})=\mu(\t(\e_n))$. Here $\mu(\t(\e_n))\p_{\t(\e_n)}(g_n)\mu(\t(\e_n))^{-1}$ lies in $\phi_{\p(\e_n)}(G(\p(\e_n)))$ if and only if $g_n\in \phi_{\e_n}(H(\e_n))$. So $\gamma$ is non-reduced. \smallskip
\item[(ii.b.3)] If $\e_n=\p(\e_n)^-$ and $\ol{\e}_{n+1}=\p(\e_n)^+$; then $\mu(\e_n)=\mu(\t(\e_n))$ and $\mu(\ol{\e}_{n+1})=1$. This leads to a contradiction since $\mu(\t(\e_n))\,\p_{\t(\e_n)}(g_n)\not\in \p_{\t(\e_n)}(H(\t(\e_n)))$ while $\phi_{\p(\e_n)}(G(\p(\e_n)))\subset \p_{\t(\e_n)}(H(\t(\e_n)))$. \smallskip
\item[(ii.b.4)]  If $\e_n=\p(\e_n)^+$ and $\ol{\e}_{n+1}=\p(\e_n)^-$; one obtains a contradiction as in the latter case.
\end{itemize} 
\end{itemize}
\end{itemize}
This concludes the proof.\smallskip
\end{proof}

{\noindent\sl$\bullet$\ \begin{minipage}[t]{\largeur}
{Construct graph of groups decompositions $\Mg$ of $\pi_1(\M)$ and $\Ng$ of $\pi_1(\N)$ and  the covering of graphs of groups $\pp:\Ng\lra\Mg$. }
\medskip
\end{minipage}
}

The choice of a maximal tree ${\rm T}$ in the underlying graph ${\rm X}$ of a graph of group $\X$ defines embeddings of the vertex and edge groups in $\pi_1(\X,\x)$. Let $\v\in {\rm VX}$, the monomorphism $G(\v)\longrightarrow \pi_1(\X,\x)$ is defined by:
$$
g\in G(\v) \longmapsto (1,\e_1,1,\ldots,\e_n,g,\ol{\e}_{n},\ldots,1,\ol{\e}_1,1).
$$
where $(\e_1,\ldots,\e_n)$ is the simple path in ${\rm T}$ from $\x$ to $\v$. Once embeddings of the vertex groups are given, their images in $\pi_1(\X,\x)$ are called {\sl vertex subgroups}. Since edge groups embed in vertex groups, embeddings of vertex groups define also embeddings of the edge groups in $\pi_1(\X,\x)$; their image in $\pi_1(\X,\x)$ are called {\sl edge subgroups} and they all lie in  vertex subgroups. For $\v\in {\rm VX}$ and $\e\in{\rm EX}$, the corresponding vertex and edge subgroups will be denoted by $G_\v$, $G_\e$.

\begin{lem}
One can  construct maximal trees ${\rm T}_{\rm N}$ of ${\rm N}$ and ${\rm T}_{\rm M}$ of ${\rm M}$ such that  $\forall\,\e\in {\rm T}_{\rm N}$, $\p(\e)\in {\rm T}_{\rm M}$.
\end{lem}

\begin{proof}
Apply a usual algorithm to construct a maximal tree ${\rm T}_{\rm N}$ of ${\rm N}$: initially  ${\rm T}_{\rm N}$ is reduced to a vertex of ${\rm N}$; while ${\rm VT}_{\rm N}\not={\rm VN}$ add to ${\rm T}_{\rm N}$  some edges $\e,\ol{\e}$ and the vertex $\v$ such that $\v\not\in {\rm VN}\setminus {\rm VT}_{\rm N}$ and $\t(\e)=\v$, $\t(\ol{\e})\in {\rm VT}_{\rm N}$. Adapt this algorithm to the search of ${\rm T}_{\rm M}$ so that $\forall\,\e\in {\rm ET}_{\rm N}$, $\p(\e)\in {\rm ET}_{\rm M}$:   initially  ${\rm T}_{\rm M}$ is reduced to a vertex of ${\rm M}$; while ${\rm VT}_{\rm M}\not={\rm VM}$ add to ${\rm T}_{\rm M}$  some edges $\e,\ol{\e}$ and the vertex $\v$ where $\e\in \p({\rm ET}_{\rm N})$, $\v\not\in {\rm VM}\setminus {\rm VT}_{\rm M}$ and $\t(\e)=\v$, $\t(\ol{\e})\in {\rm VT}_{\rm M}$. One verifies immediately that the algorithm produces a maximal tree ${\rm T}_{\rm M}$ with the required property. \smallskip
\end{proof}

{\noindent\sl$\bullet$\ \begin{minipage}[t]{\largeur}
{ Construct maximal trees ${\rm T}_{\rm M}$ of ${\rm M}$ and ${\rm T}_{\rm N}$ of ${\rm N}$ as above and fix $\x\in {\rm VM}$ and $\tl{\x}\in\p^{-1}(\x)$; that defines the vertex and edge subgroups of $\pi_1(\Mg,\x)$ and $\pi_1(\Ng,\tl{\x})$}.
\end{minipage}
}

Now that maximal trees ${\rm T}_{\rm M},{\rm T}_{\rm N}$ of ${\rm M},\,{\rm N}$ and base-points  $\x\in {\rm VM}$, $\tl{\x}\in {\rm VN}$ are given one can talk of {\sl Seifert vertex subgroups} and {\sl non-Seifert vertex subgroups} of $\pi_1(\Ng,\tl{\x})$ (respectively as those that come from puntured Seifert fibered pieces, and those that don't) and similarly of {\sl $\{1\}$}, {\sl $\ZZ$}, {\sl $\Z_2$} and {\sl $\Z\rtimes\Z$, edge subgroups} of $\pi_1(\Mg,\x)$ and (for the two former) of $\pi_1(\Ng,\tl{\x})$.
One also partition vertex subgroups of $\pi_1(\Mg,\x)$ into Seifert and non-Seifert vertex subgroups, accordingly to the partition of vertex subgroups of $\pi_1(\Ng,\tl{\x})$.
\medskip

In a graph of group $\X$, given a maximal tree ${\rm T}$ of ${\rm X}$, an $(\X,\x)$-loop $\gamma$ is said to be {\sl cyclically reduced} whenever:
\begin{itemize}
\item[(i)] $\gamma$ is a reduced $(\X,\x)$-loop, and
\item[(ii)] either its length is less than 2 or:
$$
\gamma=(1,\e_1,\ldots,1,\e_p,1)(g_0,\e_{p+1},\ldots,\e_n,g_n)(1,\ol{\e}_p,1,\ldots,\ol{\e}_1,1)
$$
where $(\e_1,\ldots,\e_p)$ is the simple path in ${\rm T}$  from $\x$ to $\t(\e_n)$ (eventually reduced to $(\x)$)
and either $\e_{p+1}\not=\ol{\e}_n$ or $g_ng_0\not\in \phi_{\e_n}(G(\e_n))$.\smallskip
\end{itemize}

We can now state basic algorithms that help working with elements in $\pi_1(\Mg)$ and $\pi_1(\Ng)$.

\begin{lem}[Basic algorithms in $\pi_1(\Mg)$]\label{reduce}
Let $\pp:\Ng\longrightarrow \Mg$ be the covering found above. Fix elements $\x\in {\rm VM}$ and $\tl{\x}\in\p^{-1}(\x)\subset {\rm VN}$; then:
\begin{itemize}
\item[(i)] {\rm (Cyclic reduction)}. There is an algorithm which given a \linebreak $(\Mg,\x)$-loop $\g$ change it into a cyclically reduced $(\Mg,\x)$-loop $\g'$, such that $\g,\g'$ represent conjugate elements in $\pi_1(\Mg,\x)$.\smallskip
\item[(ii)] {\rm (Algorithm $\mathfrak{GWP}(H,G)$)}. There is an algorithm which given a $(\Mg,\x)$-loop $\g$ decides whether $\g$ represents an element of $\pi_1(\Mg,{\x})$ lying in $\p_*(\pi_1(\Ng,\tl{\x}))$, and if so, constructs a $(\Ng,\tl{\x})$-loop  $\g'$, with same the length as $\g$, and such that $\p_\#(\g')=\g$. Moreover, whenever $\g$ is reduced (resp. cyclically reduced) then so is $\g'$.
\end{itemize}
\end{lem} 

\begin{proof} We prove separately (i) and (ii). \\
{\it Proof of {\rm (i)}}. The first step changes $\g$ into a reduced $(\Mg,\x)$-loop which represents  the same element of $\pi_1(\Mg,\x)$; this is done by applying the algorithm given by Lemma \ref{algopieces}.(iv) for the generalized word problem in edge subgroups of some vertex groups. If the reduced $(\Mg,\x)$-loop obtained, say $\g=(g_0,\e_1,\ldots , \e_n,g_n)$, has length $n< 2$, or if $\e_1\not=\ol{\e}_n$ then $\g$ is cyclically reduced and the process stops. Otherwise, $n\geq 2$ and $\e_1=\ol{\e}_n$; use Lemma \ref{algopieces}.(iv) to decide whether $g_ng_0\in \phi_{\e_n}(G(\e_n))$. If not then the $(\Mg,\x)$-loop obtained is cyclically reduced and the process stops; if yes  change it into:
$$
\g'=(1,\e'_1,\ldots, 1,\e'_p)
\underbrace{
(\phi_{\e_1}(g_ng_0)g_1,\e_2,g_2,\ldots,\e_{n-1},g_{n-1})
}_{\g''}
(\ol{\e}'_p, 1,\ldots,\ol{\e}'_1,1)
$$
where $(1,\e'_1,\ldots, 1,\e'_p)$ is the simple path in ${\rm T}$ from $\x$ to $\t(\e_1)$;  $\g''$ is a reduced $(\Mg,\t(\e_1))$-loop. Consider  the $(\Mg,\x)$-loop $\a=(1,\e'_1,\ldots,1,{\e'_p},$ $1,\ol{\e}_1,g_n)$, then $\g'=\a\g\a^{-1}$ in $\pi_1(\Mg,\x)$.   If $\g''$ is cyclically reduced then $\g'$ is cyclically reduced and the process stops. Otherwise apply the same process to the $(\Mg,\t(\e_1))$-loop $\g''$, and so on; after an eventual reduction of the prefix path in ${\rm T}$ one finally obtains a cyclically reduced $(\Mg,\x)$-loop which represents a conjugate of $\g$ in $\pi_1(\Mg,\x)$.  \smallskip

\noindent
{\it Proof of {\rm (ii)}}. For each vertex group $H(\v)$ of $\Ng$ one considers its image $\p_\v(H(\v))$ in $G(\p(\v))$, denoted $H(\p(\v))$, that one identifies with $H(\v)$; it has  index at most 2. We consider the generating sets $S_{\v}$ of $G(\v)$ and $S'_\v$ of $H(\v)$ as in Lemma \ref{algopieces}.(i). For each $\v\in {\rm VM}$, and for each  word $w$ on $S_{\v}\cup S_{\v}^{-1}$ defining an element in $H(\v)\subset G(\v)$  denote by $\ol{w}$ the word on $S'_\v\cup {S'}_\v^{-1}$, given by Lemma \ref{algopieces}.(ii),  equal to $w$ in $G(\v)$.

As in the proof of Lemma \ref{graphsofgroups}, let ${\rm VM}={\rm VM}^+\sqcup {\rm VM}^-$, ${\rm VN}=\{\v^+;\v\in {\rm VM}\}\sqcup\{\v^-;\v\in {\rm VM}^+\}$ and elements $\mu(\e)\in G(\t(\p(\e))$ for all $\e\in {\rm EN}$ defined by the covering $\pp:\Ng\lra\Mg$, with $\tl{\x}=\t(\v^+)$ whenever $\x=\t(\v)$. 
Given a $(\Mg,\x)$-loop:
$$\g=(g_0,\e_1,\ldots , \e_n,g_n)$$
where each $g_i\in G(\v_i)$ is given by a word on $S_{\v_i}\cup S_{\v_i}^{-1}$, one algorithmically change it into a $(\Ng,\tl{\x})$-loop by applying the following transformation rules from the left to the right:
$$
\begin{array}{ll}
\multicolumn{2}{l}{- \text{If}\ n=0: \forall\,w_0\ \text{a word on}\ S_{\x}\cup S_\x^{-1}}\\
w_{0} & \longmapsto \begin{cases}
\ol{w}_{0}\qquad\qquad\qquad\qquad\qquad\quad\ \, &\ w_0\in H(\x)\\
w_0 &\ w_0\not\in H(\x)
\end{cases}\smallskip\\
%
\multicolumn{2}{l}{- \text{If}\ n>0:}\\
\multicolumn{2}{l}{\forall\,w_i\ \text{a word on}\ S_{\t(\ol{\e}_{i+1})}\cup S_{\t(\ol{\e}_{i+1})}^{-1}}\smallskip\\
 w_i,\e_{i+1},w_{i+1} & \longmapsto \begin{cases}
\ol{w}_i,\e_{i+1}^+,w_{i+1} &\ w_i\in H(\t(\ol{\e}_{i+1}))\\
\ol{w_i\mu(\ol{\e}_{i+1}}),\e_{i+1}^-,\mu(\e_{i+1})^{-1}w_{i+1} &\ w_i\not\in H(\t(\ol{\e}_{i+1}))
\end{cases}\\
\multicolumn{2}{l}{\forall\,w_n\ \text{a word on}\ S_{\t({\e}_n)}\cup S_{\t({\e}_n)}^{-1}}\smallskip\\
\e_n^\pm,w_{n} & \longmapsto \begin{cases}
\e_{n}^\pm,\ol{w}_{n}\qquad\qquad\qquad\qquad\quad\ \ \,&\ w_n\in H(\x)\\
\e_{n}^\pm,w_n &\ w_n\not\in H(\x)
\end{cases}
\end{array}
$$
If one denotes by $\g'=(g'_0,\e^\pm_1,\ldots , \e^\pm_n,g'_n)$ the loop obtained, then $\g'$ is a $(\Ng,\tl{\x})$-loop if and only if $g'_n\in H(\t(\e_n))$ if and only if $\g$ represents a element in $\p_*(\pi_1(\Ng,\tl{\x}))$, and in such case $\p_\#(\g')=\g$.
In particular, since $p_\#$ induces an homomorphism $\p_*:\pi_1(\Ng,\tl{\x})\lra\pi_1(\Mg,\x)$, when $\g'$ is not reduced then neither is $\g$. The same argument applied in $\pi_1(\Mg,{\rm y})$ for ${\rm y}=\t(\e_1)$ shows that whenever $\g$ is cyclically reduced then so is $\g'$. 
\end{proof}

{\noindent\sl$\bullet$\ \begin{minipage}[t]{\largeur}
{
In the following we usually write:
$$G=\pi_1(\Mg,\x)\quad ;\quad H=\pi_1(\Ng,\tl{\x})$$
with $H$ identified with the subgroup $\p_*(H)$ of $G$.}
\end{minipage}
}

We now make a review of what is known on basic Dehn problems in $G$ and $H$.

\begin{lem} The following algorithmic problems are known to be solvable  in $G$ and $H$:\label{dehnpbm}
\begin{itemize}
\item[(i)] {\rm (Algorithms $\mathfrak{WP}(H)$, $\mathfrak{CP}(H)$):} the word and conjugacy problems in $H$.
\item[(ii)] {\rm (Algorithm $\mathfrak{WP}(G)$):} the word problem in $G$.
\item[(iii)] {\rm (Algorithms $\mathfrak{WP}(H(\v))$, $\mathfrak{CP}(H(\v))$:} the word and conjugacy problems in vertex subgroups of $H$.
\item[(iv)] {\rm (Algorithms $\mathfrak{WP}(G(\v))$, $\mathfrak{CP}(G(\v))$:} the word and conjugacy problems in vertex subgroups of $G$.
\end{itemize}
\end{lem}

\begin{proof} (i). By hypothesis the 3-manifold $\N$ is orientable and geometrisable. A solution to the conjugacy problem in $H$ is given in \cite{cp3mg}. Let's give some details for a solution to the word problem.  Note that $H$ is a finitely presentable group. According to \cite[Theorem 12.4.7]{epstein} either $H=\pi_1(\N)$ is automatic or  $\N$ is closed and modeled on one of $Nil$, $Sol$  geometries. Automatic groups have solvable word problems (in quadratic time, cf. 
\cite[Theorem 2.3.10]{epstein}). 
A finitely presented group that contains a finite index subgroup with solvable word problem also has a solvable word problem; indeed their Dehn functions are equivalent (\cite[Proposition 1.3.4 and Exercise 1.3.5]{bridson}); note also that the finite index subgroup is also finitely presentable.
When $\N$ is modeled on $Nil$ then $H$ is virtually nilpotent, and finitely generated nilpotent groups have solvable word problem (in linear time, cf. \cite{goodmanshapiro}). When $\N$ is modeled on $Sol$ then $H$ is virtually $\ZZ\rtimes _\phi\Z$ with $\phi$ Anosov (cf. \cite[Theorem 5.3.(i)]{scott}) where a solution to the word problem  is easily provided by rewriting any element as a couple in $\ZZ\times\Z$. In any case $H$ has a solvable word problem.

(ii). Since $G$ is finitely presentable and contains a finite index subgroup $H$ having a solvable word problem, so does $G$ by the same argument as above.

(iii). Let $H_\v$ be a vertex subgroup of $H$; 
 in such case the solution is provided by \cite{epstein,nr1} or \cite{cp3mg} depending on whether $H_\v$ is biautomatic or arises from a piece modeled on $Nil$-geometry.

 (iv). Let $G_\v$ be a vertex subgroup of $G$; in cases it comes from an orientable piece the same argument  as in (iii) applies. Otherwise: a  non-orientable 3-manifold cannot be modeled on $Nil$-geometry (cf.\cite{scott}); it follows from \cite{epstein,nr1} that $G_\v$ is biautomatic and hence has solvable word and conjugacy problems (\cite[Theorem 2.5.7]{epstein}).
\end{proof}


We will make a heavy use of these basic algorithms. Another main ingredient will be that following algorithm finding the centralizer of any element in $H$. In case of Haken orientable 3-manifolds the structure of centralizers are quite simple and related to the JSJ decomposition as stated in  \cite[Theorem VI.I.6]{js}; one deduces the centralizers in groups of geometrisable orientable 3-manifolds, that one can compute, as follows.

\begin{lem}[Algorithm $\mathfrak{ZP}(H)$ for centralizers in $H$]\label{centralizer}
Let $u\in H\smallsetminus \{1\}$ as above; then exactly one of the following assertions occurs:
\begin{itemize}
\item[(i)] its centralizer $Z_H(u)$ is infinite cyclic and  does not lie in the conjugate  of a Seifert vertex subgroup,
\item[(ii)] $Z_H(u)$ lies in the conjugate of a Seifert vertex subgroup $H_\v$,
\item[(iii)] $Z_H(u)$ is conjugate to a $\ZZ$ edge subgroup $H_\e$ and does not lie in the conjugate of a Seifert vertex subgroup.\smallskip
\end{itemize}
There is an algorithm that determines which case occurs, and in cases (ii) and (iii) produces all possible vertex or edge subgroups $H_\v$, $H_\e$  and conjugating elements.
\end{lem}

\begin{proof} 
First note that in case of an empty tori decomposition for $\N$,  since in groups of hyperbolic closed manifolds non-trivial centralizers  are all infinite cyclic (cf. \cite{scott}), the former assumption follows from \cite[Theorem VI.I.6]{js} and the latter assumption from Lemma \ref{seifertinvariant}. So we suppose in the following that the tori decomposition of $\N$ is non-empty.\smallskip

We are given an element $u\in H$ by a $(\Ng,\tl{\x})$-loop and are interested in its
centralizer; first apply the
algorithm of Lemma \ref{reduce}.(i) to $\p_\#(u)$ followed by that one of Lemma \ref{reduce}.(ii) to change $u$ into a cyclically reduced   $(\Ng,\tl{\x})$-loop $v$ conjugate to $u$ in $H$; it finds $h\in H$ such that $huh^{-1}=v$. Obviously $Z_H(v)=h\,Z_H(u)\,h^{-1}$, so that in the following we suppose that $u$ is a cyclically reduced $(\Ng,\tl{\x})$-loop.\smallskip

Let ${\rm EN}_0$ be the subset of ${\rm EN}$ of those edges whose groups are all trivial ({\it i.e.} related to spheres  in the topological    decomposition of $\N$). Denote $N_1,\ldots N_p$ the connected components of the graph obtained from $N$ by deleting all edges in ${\rm EN}_0$. Together with $\Ng$ this defines graphs of groups $\Ng_1,\ldots \Ng_p$ by restricting $\Ng$ to the respective subgraphs ${\rm N}_1,\ldots,{\rm N}_p$. The choice of base points $\tl{\x}_1,\ldots,\tl{\x}_p$ in ${\rm N}_1,\ldots,{\rm N}_p$ together with the maximal tree ${\rm T}_{\rm N}$ defines natural monomorphisms from $\pi_1(\Ng_1,\tl{\x}_1),\ldots,\pi_1(\Ng_p,\tl{\x}_p)$ into $H=\pi_1(\Ng,\tl{\x})$.
Moreover $H$ splits as the free product $\pi_1(\Ng,\tl{\x}_1)\ast\cdots\ast\pi_1(\Ng_p,\tl{\x}_p)*F_n$, for $F_n$   a free group of finite rank. Note that $\Ng_1,\ldots,\Ng_p$ are graphs of groups related to the decompositions of the $\widehat{\N}_1$,$\ldots,\widehat{\N}_p$ along the tori appearing in  Lemma \ref{JSJ}, and that the images of $\pi_1(\Ng_1,\tl{\x}_1),\ldots,\pi_1(\Ng_p,\tl{\x}_p)$ into $H=\pi_1(\Ng,\tl{\x})$ coincide up to conjugacy with those induced by the inclusions of $\N_1,\ldots,\N_p$ into $\N$.

One obtains from $u$ a cyclically reduced sequence with respect to the free product. If its length is positive then $Z_H(u)$ is infinite cyclic, the case (i) occurs; otherwise $Z_H(u)$ lies in one of the free product factors (\cite[Corollary 4.1.6]{mks}). In the latter case it follows from \cite[Theorem VI.1.6]{js} and \cite[characteristic pair Theorem]{js} that exactly one of the assertions (i), (ii) or (iii) occurs; this proves the first assumption.\smallskip

We now return to the algorithm: if $u$ passes through an edge in ${\rm EN}_0$, then the case $(i)$ occurs and otherwise $Z_H(u)$ is included  in some factor, say $\pi_1(\Ng_1,\tl{\x_1})$, that one can decide from $u$. Suppose in the following that $u$ lies in $K\triangleq\pi_1(\Ng_1,\tl{\x}_1)$, and write $u$ as a cyclically reduced $(\Ng_1,\tl{\x}_1)$-loop.

 The
algorithm is constructed on procedures and arguments stated in \cite{cp3mg}, which apply here since $\Ng_1$ is related to the tori decomposition of  the  orientable irreducible geometrisable 3-manifold $\widehat{\N}_1$. 
Suppose that $u$ is conjugate in $K$ to an element $u'$ in a vertex subgroup $H_{\v'}$, for some $\v'\in {\rm VN}_1$. According to \cite[Theorem 3.1]{cp3mg}, since $u$ is cyclically reduced  one of the following cases occurs in $K=\pi_1(\Ng_1,\tl{\x}_1)$:
\begin{itemize}
\item[(i)] $u$ lies in $H_{\v'}$ and $u$, $u'$ are conjugate in $H_{\v'}$, or
\item[(ii)] $u$ lies in a vertex subgroup $H_{\v}$, $\v\in {\rm VN}_1$, and there is a sequence $(c_1,\ldots,c_n)$ of elements of edge subgroups such that $u$ is conjugate to $c_1$ in $H_{\v}$, $u'$ is conjugate to $c_n$ in $H_{\v'}$ and for any $i=1\ldots n-1$, either $(\e,c_i,\ol{\e})=c_{i+1}$ for some $\e\in {\rm EN}_1$ or $c_i$ and $c_{i+1}$ are conjugate in some vertex subgroup.   
\end{itemize}

Let's return to the algorithm;
 given  a cyclically reduced $(\Ng_1,\tl{\x}_1)$-loop $u$ one can decide whether  $u$ lies  in a vertex subgroup. If not then with the above, $u$ is not conjugate to a vertex subgroup and assertion (i) occurs: $Z_H(u)$ is infinite cyclic.
So in the following we will suppose that  $u$ lies in a vertex subgroup $H_{\v}$  of $K$, for some $\v\in {\rm EN}_1$.\smallskip

First consider the particular case where $\widehat{\N}_1$ is a $\TT^2$-bundle over $\SS^1$ modeled on $Sol$ (cf. Theorem 5.3, \cite{scott}). This occurs when ${\rm N}_1$ is a cycle with one or two vertices (resp. edges) and all vertex and edge groups are free abelian with rank 2. In such case  $K$ splits as $(\ZZ)\rtimes_\theta\Z$ for some $\theta\in SL_2\Z$ Anosov and the left factor coincides with all vertex and edge subgroups of $K$. It follows easily from the fact that $\theta$ has no eigenvalue with modulus 1 that the centralizer of any element in $K$ is either infinite cyclic or consists in the whole left factor $\ZZ$.  Hence assertion (ii) occurs.\smallskip

Now consider the remaining cases when $\widehat{\N}_1$ is not a $\TT^2$-bundle over $\SS^1$ modeled on $Sol$.
Using the Seifert invariants obtained by the algorithm $\mathfrak{Top4}$ (Lemma \ref{seifertinvariant}) one  decides which vertex subgroup is a Seifert subgroup and among them which comes from a $\TT^2\times \I$ piece (those with basis an annulus and no exceptional fiber); note that the latter correspond to vertex subgroups which are free abelian with rank 2  (Theorem 10.5, \cite{hempel}).
Use the following process to find all elements in vertex subgroups conjugate to $u$ in $H_{\v}$:\\
-- If $H_{\v}$ is not a Seifert vertex subgroup; then according to \cite[Proposition 4.1]{cp3mg}, $u$ is conjugate in $H_{\v}$ to at most 1 element  lying in at most one edge subgroup, that using  \cite[Theorem 6.3]{cp3mg}, one finds together with a conjugating element in $H_{\v}$.\\
-- If $H_\v\simeq \ZZ$;  $u$ lies both in the two edge subgroups and is not conjugate in $H_\v$ to any other element.\\
--  If $H_\v$ is  a Seifert vertex subgroup and $H_\v\not\simeq\ZZ$. According to  \cite[Proposition 4.1]{cp3mg}, 
either $u$ is conjugate in $H_\v$ to at most 1 element lying in at most one edge subgroup, or $u$ lies in a {\sl fiber} of $H_\v$, {\it i.e.}  is a power of a regular fiber in a Seifert  fibration of the corresponding piece and $u$ lies in the intersection of all edge subgroups  in $H_\v$. One decides using  \cite[Proposition 5.1]{cp3mg} (note also that deciding whether $u$ lies in a fiber of $H_\v$  is easily done by checking with a solution to the word problem whether for all generators $s$ of $H_\v$, $sus^{-1}=u^{\pm 1}$, cf.  \cite[Lemma II.4.2.(i)]{js}).

Pursue the process with the successive conjugates in the edge subgroups obtained, the acylindricity of $\Ng_1$ (\cite[Lemma 4.1]{cp3mg}) ensures that it finally stops and one finally obtains a finite list of all elements in vertex and edge subgroups to which $u$ is conjugate, as well as conjugating elements.

The minimality of the JSJ decomposition of $\widehat{\N}_1$ ensures that $u$ is neither conjugate to the fibers of two Seifert vertex subgroups $\not\simeq\ZZ$ nor to the fibers of two Seifert vertex subgroups $\simeq\ZZ$. Then with the above, together with  \cite[Theorem VI.I.6]{js}, one finally obtains exactly one of the following cases:
\begin{itemize}
\item[--] $u$ is conjugate neither to a Seifert vertex subgroup nor to an edge subgroup: $Z_H(u)\simeq\Z$ and assertion (i) occurs.\smallskip
\item[--] $u$ is conjugate to a Seifert vertex subgroup $H_{\v'}$ and is not conjugate to a fiber of any Seifert vertex subgroup $\not\simeq\ZZ$. In that case $Z_H(u)$ lies up to conjugacy in $H_{\v'}$ and  assertion (ii) occurs.\smallskip
\item[--] $u$ is conjugate to the fiber of a Seifert vertex subgroup $H_{\v'}\not\simeq \ZZ$; $Z_H(u)$ lies up to conjugacy in $H_{\v'}$ and assertion (ii) occurs.\smallskip
\item[--] $u$ does not lie in the conjugate of any Seifert vertex subgroup but lies in the conjugate of an edge subgroup $H_\e$; $Z_H(u)=H_\e$ and assertion (iii) occurs.
\end{itemize}
This achieves the proof.
\end{proof}
\subsection{Step 4: The  conjugacy algorithm.}
 We construct in this section the algorithm that solves conjugacy problem in $G$. Whenever  $u=hvh^{-1}$  we shall use the notation $u=v^h$ or $u\sim v$.\smallskip

{\noindent\sl$\bullet$\ \begin{minipage}[t]{\largeur}
Suppose that $u$ and $v \in G$ are given by a couple of $(\Mg,\x)$-loops and one wants to decide whether $u\sim v$ in $G$.
\medskip
\end{minipage}
 }

 First use the solution $\mathfrak{GWP}(H,G)$ (Lemma \ref{reduce}.(ii)) to the generalized word problem of $H$ in $G$ to decide whether $u,v$ lie in $H$ other
 not.\smallskip

 {\noindent\sl$\bullet$ If either $u$ or $v$ lies in $H$.}\smallskip

\indent
 Since
$H$ has index 2 in $G$, if  $u$ and $v$ lie in different classes of $H/G$ they are definitely not conjugate in $G$. If
$u$ and $v$ both lie in $H$, then the solution $\mathfrak{CP}(H)$ (Lemma \ref{dehnpbm}.(i)) to the conjugacy problem in $H$ together with the following lemma  allow to
decide whether $u$ and $v$ are conjugate in $G$.
%
\begin{lem}[Algorithm $\mathfrak{CP1}(K)$]
\label{ld1}
 Let $K$ be a group and $L$  an index 2 subgroup of $K$ with solvable conjugacy problem.
 Given any couple of elements $u,v\in L$ one can decide whether $u$ and $v$ are conjugate in $K$.
\end{lem}
\begin{proof} Given a set of representative $a_0=1,a_1$ of $L /K$, in order to decide whether $u,v\in L$ are
conjugate in $K$ it suffices to check whether $u$ is conjugate in $L$ to any of the $a_iva_i^{-1}$ for $i=0,1$.
\end{proof}

 {\noindent\sl$\bullet$ In the following both $u$ and $v$ lie in $G\smallsetminus H$ and are supposed to be cyclically reduced.}\medskip

\indent
This can be done by applying the algorithm in Lemma \ref{reduce}.(i) to change $u$ and $v$ into two cyclically reduced $(\Mg,\x)$-loops, respective conjugates of $u$, $v$ in $G$.\smallskip

Decide whether $u$, $v$ have order 2; according to \cite{serre} it occurs when $u,v$ lie in vertex subgroups of $G$, so use a solution to the word problem in those vertex subgroups (Lemma \ref{dehnpbm}.(iv)) to check whether $u^2=1$, $v^2=1$ (or $\mathfrak{WP}(G)$, Lemma \ref{dehnpbm}.(ii)).
 If exactly one of the
relations occurs then $u$ and $v$ are not conjugate in $G$.\medskip
\noindent
{\sl$\bullet$ If both $u$ and $v$ have order 2.}\medskip

\indent
 In such case the following lemma
allows to decide whether $u$ and $v$ are conjugate other not.

\begin{lem}[Algorithm $\mathfrak{CP2}(G)$]
\label{ordre2} One can  decide for any pair
of order 2 elements $u, v\in G$ whether $u$ and $v$ are conjugate
in $G$.
\end{lem}
\begin{proof}  Recall  the system $\cal P$ of essential projective planes in $\M$ as in \S \ref{section:topological} ; they are necessarily 
pairwise non parallel. It follows from \cite{epsteinz2}, \cite{stalling}, \cite{swarup} that  each order 2 element in
$G$ is conjugate to some $\Z_2$-edge subgroup of $G$ and that all $\Z_2$-edge subgroups are pairwise non conjugate in $G$.

Let $u$, $v$ be cyclically reduced element of order 2 lying in respective vertex subgroups $G_\v$, $G_{\v'}$.  According to \cite{epsteinz2} (or \cite[Theorem 9.8.(i)]{hempel}) and   \cite[Proposition 2.2]{swarup}, $u$ and $v$ are necessarily conjugate in $G_\v$ and $G_{\v'}$ to the generators of the $\Z_2$-edge subgroups. One decides so using the solutions $\mathfrak{CP}(G_\v)$, $\mathfrak{CP}(G_{\v'})$ (Lemma \ref{dehnpbm}.(iv)) to the conjugacy problems in $G_\v$ and $G_{\v'}$.  Then $u$, $v$ are conjugate in $G$ if and only if they are 
conjugate to the non-trivial element in a same $\Z_2$-vertex subgroup, or to the non-trivial elements in two $\Z_2$-vertex subgroups coming from opposite edges $\e$, $\ol{\e}\in {\rm EM}$.
\end{proof}

 We will be concerned in the following with
the remaining case: $u,v$ both lie in $G\smallsetminus H$ and both have order different than 2. According to
\cite{epsteinz2} both $u$ and $v$ must have infinite order in $G$.\medskip

{\noindent\sl$\bullet$\ \begin{minipage}[t]{\largeur}
In the following both $u$ and $v$ lie in $G\smallsetminus H$ and have infinite order.\smallskip
\end{minipage}
 }

 Use algorithm $\mathfrak{CP_1}(G)$ (Lemma \ref{ld1})  to decide whether $u^2\sim v^2$ in
$G$ and find, if any, $k\in G$ that conjugates $u^2$ into $v^2$; if such $k\in G$ does not exist then $u$, $v$ are not conjugate in $G$. So we suppose in the following that $u^2\sim v^2$ in $G$ and we are given an element $k\in G$ such that $u^2=(v^2)^k$ in $G$; such a conjugating element is implicitly provided (going into the
lines of the proof) by the conjugacy algorithm in \cite{cp3mg}.\smallskip

 {\noindent\sl$\bullet$\ \begin{minipage}[t]{\largeur} In the following $u^2$ and $v^2$ are supposed to be conjugate in $G$ and we are given $k\in G$ such that
 $u^2=(v^2)^k$.\smallskip
\end{minipage}
}

We first need to fix some notations which will be useful in the following.  Denote by
$Z_G(v)=\{u\in G\ |\ uv=vu\}$ the centralizer of $v$ in $G$ and by $C_G(u,v)=\{k\in G\ |\ u=v^k\}$. The subset $C_G(u,v)$
of $G$ is either empty (when $u\not\sim v$) or equal to $k.Z_G(v)$ for any $k\in G$ such that $u=v^k$.
 The set $C_G(u^2,v^2)=k.Z_G(v^2)$ is non empty.
 It obviously
 contains the set
$C_G(u,v)$; note also that $Z_G(v^2)$ contains $Z_G(v)$ as a subgroup as well as $Z_H(v^2)$ as an index 2 subgroup;
$Z_G(v^2)$ is generated by $Z_H(v^2)$ and $v$.\smallskip

We are now interested in the centralizer $Z_H(v^2)$ of $v^2$ in $H$. Apply the algorithm $\mathfrak{ZP}(H)$ (Lemma \ref{centralizer}) to check
whether it
is infinite cyclic or  conjugate into a  Seifert vertex subgroup or to a $\ZZ$-edge subgroup of $H$. It provides, if any, the Seifert vertex subgroup $H_\v$ or edge subgroup $H_\e$ of $H$ and the conjugating element $h\in H$ such that $hZ_H(v^2)h^{-1}$ lies in $H_\v$ or $H_\e$. \medskip

{\noindent\sl$\bullet$\ \begin{minipage}[t]{\largeur} Check whether $Z_H(v^2)$ is infinite cyclic or is conjugate to a $\ZZ$-edge subgroup or into a Seifert vertex subgroup of $H$.\\
\end{minipage}}

\noindent We further treat separately the former case and the two latter cases. \medskip

{\noindent\sl$\bullet$ Case (i): $Z_H(v^2)$ is infinite cyclic.}\medskip

In that case, $Z_G(v^2)$ contains $\Z$ as an index 2
subgroup. If $Z_G(v^2)$ is torsion-free then it must be cyclic, say $Z_G(v^2)=<w>$. But since $v\in Z_G(v^2)$,  $v$ is a
power of $w$ so that $w\in Z_G(v)$, and since $Z_G(v)\subset Z_G(v^2)$, it implies that $Z_G(v)=Z_G(v^2)=<w>$. If
$Z_G(v^2)$ is not torsion-free, let us denote by $t$ a generator of its index 2 subgroup $Z_H(v^2)$. The group $Z_G(v^2)$ is
generated by $v$ and $t$ and must be one of the two groups appearing in the following lemma.
\begin{lem}[Groups with torsion containing $\Z$ as an index 2 subgroup]\label{l1}
A group $K$ with torsion and generators $v,t$, such that $<t>\simeq\Z$ has index 2 in $K$ must be one of:
$$<v,t\ |\ [v,t]=1, v^2=t^{2n}>\simeq \Z\oplus\Z_2$$
$$<v,t\ |\ t^v=t^{-1}, v^2=1>\simeq \Z_2 *\Z_2$$
\end{lem}
\begin{proof}  The group $K$ admits as presentation $<v,t\ |\ t^v=t^{\pm1}, v^2=t^p>$ for
some $p\in\Z$. The set $K\smallsetminus <t>$  contains an element $w$ with finite order $m\not=0$. In particular $w^m$ lies
in the index 2 subgroup $<t>$ so that $m$ must be even; hence $K$ contains an element $t^{-n}v$ with order 2, for some
$n\in\Z$.
Suppose first that $t^v=t$, so that $1=(t^{-n}v)^2=t^{-2n+p}$. It follows that $p=2n$ and one obtains the first
presentation.
Suppose then that $t^v=t^{-1}$; one has $1=(t^{-n}v)^2=v^2$ and one obtains the  second presentation.
\end{proof}

The latter group cannot occur since $v$ has infinite order. Concerning the former group, since $[v,t]=1$, one has
$Z_G(v)= Z_G(v^2)$. Hence whenever  $Z_H(v^2)$ is infinite cyclic then $Z_G(v)=Z_G(v^2)$ and the following lemma allows
us to decide whether $u\sim v$ in $G$.

\begin{lem} [Algorithm $\mathfrak{CP3}(K)$]
\label{ld2} Let $K$ be a group and $L$ be an index 2 subgroup of $K$. Suppose that $L$ has a
solvable conjugacy problem. Let $v\in K\smallsetminus L$ such that
 $Z_K(v)=Z_K(v^2)$. Then  one can decide for any $u\in K$ whether $u$ and $v$ are conjugate in $K$.
\end{lem}

\begin{proof} Since $L$ has solvable conjugacy problem, $L$ has solvable word problem, and hence $K$ also has solvable word problem. Let $v\in K$ be as above, and suppose one wants to decide for some given $u\in K$ whether $u
\sim v$ in $K$.
  With Lemma \ref{ld1} one can decide whether $u^2$ and $v^2$ are conjugate in $K$. If not
   then $u$ and $v$ are definitely not conjugate in $K$. So suppose that
    $u^2=kv^2k^{-1}$ for some $k\in K$ that one can effectively find using a solution to the word problem in $K$
    (in our purpose $k$ is provided by the solution of \cite{cp3mg} to conjugacy in $H$),
    so that $C_K(u^2,v^2)=kZ_K(v^2)$.
    Obviously $C_K(u,v)\subset C_K(u^2,v^2)$ and moreover since
    $Z_K(v^2)=Z_K(v)$,
if $C_K(u,v)$ is non-empty it must equal $C_K(u^2,v^2)$. Hence to decide whether $u$ and $v$ are conjugate in $K$ it
suffices to decide using the solution to the word problem in $K$ whether $u=v^k$ other not.\smallskip
\end{proof}

{\noindent\sl$\bullet$\ \begin{minipage}[t]{\largeur}
 Case (ii) and (iii): $Z_H(v^2)$ is conjugate in $H$ to a subgroup of a  Seifert vertex subgroup or to a $\ZZ$-edge subgroup of $H$.\medskip
\end{minipage}
 }

Let $h\in H$ be given by the algorithm $\mathfrak{ZP}(H)$ (Lemma \ref{centralizer}) and such that  $hZ_H(v^2)h^{-1}$ lies into a Seifert vertex subgroup $H_\v$ or an edge subgroup $H_\e$ of $H$.
Let $G_\v\supset H_\v$ and $G_\e\supset H_\e$ be the corresponding Seifert vertex subgroup or edge subgroup in $G$. Since $hZ_H(v^2)h^{-1}$, and in particular $hv^2h^{-1}$, lie in $G_\v$ or $G_\e$ one may expect that $hvh^{-1}$ also lies in $G_\v$, $G_\e$. This turns to be false in general, though only in specific cases.



%

\begin{lem}[Square root of an element of $G$ lying in a vertex or edge subgroup]\label{squareroot}
Let $v$ be a cyclically reduced element of $G\smallsetminus H$ with infinite order and $h\in H$, be such that $Z_H(v^2)^h$ lies in a vertex or edge subgroup of $G$. Then either:
\begin{itemize}
\item[(i)] $Z_H(v^2)^h$ and $v^h$ both lie in a same vertex subgroup or edge subgroup of $G$, or
\item[(ii)] $v$ lies in a vertex subgroup  which contains $\ZZ$ as an index 2 subgroup, and does not lie in  any edge subgroup.
\end{itemize} 
\end{lem}

\begin{proof}
Consider  the  cyclically reduced $(\Mg,\x)$-loop $v$, it takes one of the following forms:
$$
v=(1,\e_1,1,\ldots,\e_n,1)(v_n,\e_{n+1},\ldots,\e_m,v_m)(1,\ol{\e}_n,\ldots,1,\ol{\e}_1,1)
$$
or $v=(v_1,\e,v_2)$, or $v=(v_0)$. Then $v^2$ takes  one of the forms:
$$
\begin{array}{l}
v^2=(1,\e_1,1,\ldots,\e_n,1)(v_n,\e_{n+1},\ldots,\e_m,v_mv_n,\e_{n+1},\ldots,\e_m,v_m)\\
\qquad\qquad\qquad\qquad\qquad\qquad\qquad\qquad\qquad\qquad\quad(1,\ol{\e}_n,\ldots,1,\ol{\e}_1,1)
\end{array}
$$
or $v^2=(v_1,\e,v_2v_1,\e,v_2)$, or $v^2=(v_0^2)$, which are cyclically reduced except possibly in the first case when $m=n$. Hence if $v^2$ lies in a vertex (or edge) subgroup, then $v$ also lies in a vertex subgroup (not necessarily the same) of $G$.

Write, using  the algorithm  of Lemma \ref{reduce}.(ii), $v^2$ into a cyclically reduced $(\Ng,\tl{\x})$-loop;  $v^2$ is conjugate to a vertex or edge subgroup of $H$ and according to  \cite[Theorem 3.1]{cp3mg} $v^2$ lies in a vertex subgroup $H_\v$ of $H$ and is conjugate in $H_\v$ to one of its edge subgroup $H_\e$.  Hence $v,v^2$ both lie in a vertex subgroup $G_\v$ of $G$ and $v^2$ is conjugate in $G_\v$ to an edge subgroup $G_\e$. 
Moreover
if whenever $v,v^2$ both lie in a conjugate $wG_\v w^{-1}$ of a vertex subgroup and $zv^2z^{-1}$, for some $z\in wG_\v w^{-1}$, lies in $wG_\e w^{-1}$, for $G_\e$ an edge subgroup  of $G_\v$, one has that $zv z^{-1}$ also lie in $wG_\e w^{-1}$, then  necessarily $\forall\, g\in G$, if $gv^2g^{-1}$ lies in a vertex or edge subgroup of $G$, then $gvg^{-1}$ lies in the same vertex or edge subgroup. In particular assumption (i) holds.\smallskip

So suppose in the following that, up to conjugacy, $v,v^2$ both lie in a vertex subgroup $G_\v$, and  that for some $z\in H_\v$, $zv^2z^{-1}$ lies in an edge subgroup $G_\e$ of $G_\v$ while $zvz^{-1}$ does not. 
Since $v\not\in H$, the piece $\M_\v$ in the topological decomposition of $\M$ associated to the vertex $\v$ is non-orientable; let $\N_\v$ be its orientation cover. Since $v$ has infinite order the edge $\e$ can be associated to either $\TT^2$ or $\KB^2$. \smallskip\\
{\it First case: $\e$ is associated to $\TT^2$.} Then $\N_\v$ has two $\TT^2$ in its boundary which give rises up to conjugacy to two edge subgroups $G_\e$ and $vG_\e v^{-1}$  in $H_\v$. But since $zv^2z^{-1}\in G_\e$,  $v^2$ lies in both conjugates  of  $G_\e$ and $vG_\e v^{-1}$ in $H_\v$ (with respective conjugating elements $z^{-1}$ and $v{z^{-1}}v^{-1}$). Hence by \cite[Proposition 4.1]{cp3mg}, $\widehat{\N_\v}$ is a Seifert fiber space and $v^2$ is a power of a regular fiber of a Seifert fibration. If $H_\v\not\simeq \ZZ$ then (see the end of the proof of Lemma \ref{centralizer}) $Z_H(v^2)\subset H_\v$ and $Z_G(v^2)\subset G_\v$; assumption (i) occurs. Otherwise $v$ lies in  a vertex subgroup containing $\ZZ$ as an index 2 subgroup.\smallskip

\noindent
{\it Second case: $\e$ is associated to $\KB^2$.} Then $\N_\v$ has one $\TT^2$ in its boundary which gives rise to the vertex subgroup $H_\e=G_\e\cap H$ of $H$. Since $zv^2z^{-1}\in H_\e$, one has that $v^2$ lies in both conjugates of $H_\e$ and $vH_\e v^{-1}$  in $H_\v$ (with respective conjugating elements $z^{-1}$ and $vz^{-1}v^{-1}$). Let $w\in G_\v$ such that $H_\e$ and $w$ generate $G_\e$; there exists $t\in H_\v\smallsetminus H_\e$ such that $v=tw$. Since $wH_\e w^{-1}=H_\e$ one has also that $v^2$ lies in two conjugates of $H_\e$ in $H_\v$ (with respective conjugating elements $z^{-1}$ and $vz^{-1}v^{-1}t$). 
Since $zvz^{-1}\not\in G_\e$, necessarily $zvz^{-1}v^{-1}t\not\in H_\e$ and it follows from \cite[Proposition 4.1]{cp3mg} that $v^2$ is a power of a regular fiber in a Seifert fibration of $\widehat{\N}_\v$.
The conclusion follows as in the first case.\smallskip

By taking successive conjugates of $v$ in adjacent vertex subgroups of $G$, one finally obtains that either assumption (i) occurs or at some stage one obtains a conjugate of $v$ which lies  in a vertex subgroup $G_\v$ containing $\ZZ$ as an index 2 subgroup, and does not lies in some conjugate in $G_\v$ of an edge subgroup $G_\e$ that contains $\ZZ$. Note  using the following Lemma, that in such case, since on the one hand such edge subgroup is normal and on the other $v$ has infinite order, then $v$ is not conjugate in $G_\v$ to any of the edge subgroups of $G_\v$. 
So finally, in case one of the successive conjugates of $v$ lies in such a vertex subgroup $G_\v$ containing a $\ZZ$ of index 2, then this can arise only at initial step, that is  $v\in G_\v$: in such case  conclusion (ii) occurs.
\end{proof}

\begin{lem}[Description of the non-orientable pieces whose groups contain $\ZZ$ as an index 2 subgroup]\label{virtualzz}
Let $G_\v$ be a vertex subgroup of $G$ which is not included in $H$ and which contains $\ZZ$ as an index 2 subgroup . Let $\M_\v$ be the 3-manifold associated to the vertex $\v$ of $\Mg$ and let $St(\v)=\{\e\in {\rm EM}, \t(\e)=\v\}$. Then exactly one of the following cases occurs:\smallskip
\begin{itemize}
\item[(1)] $\M_\v$ is the product of a Moebius band with $\SS^1$, $\partial\M_\v$ consists of one $\TT^2$, $St(\v)=\{\e\}$,
$$
G_\v=<a,b,t\,|\,[a,b]=[b,t]=1,\,  t^2=a>\ \simeq\ \ZZ\\
$$
and $G_e\simeq\ZZ$  is generated by $a$ and $b$.\smallskip
\item[(2)] $\M_\v$ is homeomorphic to $\KB^2\times\I$, $\partial\M_\v$ consists of two $\KB^2$, $St(\v)=\{\e_1,\e_2\}$,
$$
G_\e=<a,b,t\,|\,[a,b]=1,\, b^t=b^{-1},\, t^2=a>\ \simeq\ \Z\rtimes\Z\\
$$
and $G_{\e_1}=G_{\e_2}=G_\v$.\smallskip
\item[(3)] $\partial\M_\v$ consists of one $\TT^2$ and four $\PP^2$; $St(\v)=\{\e_0,\e_1,\e_2,\e_3,\e_4\}$,
$$
<a,b,t\,|\,[a,b]=1,\, a^t=a^{-1},\, b^t=b^{-1},\, t^2=1>\ \simeq (\ZZ)\rtimes_{-I}\Z_2
$$
$G_{\e_0}=<a,b>\simeq\ZZ$ and $G_{\e_i}\simeq\Z_2$, $i=1\ldots 4$, are generated respectively by: $t$, $at$, $bt$, $abt$. 
\end{itemize}
Moreover, two elements $u=a^{n_1}b^{m_1}t$ and $v=a^{n_2}b^{m_2}t$ are conjugate in $G_\v$ if and only if, respectively:
$$
\begin{array}{lll}
(1)\left\lbrace\begin{array}{l}
n_1=n_2\\ m_1=m_2\end{array}\right.; &
(2)\left\lbrace\begin{array}{l}
n_1=n_2\\  m_1=m_2\ {\rm mod}\ 2\end{array}\right.; &
(3)\left\lbrace\begin{array}{l}
n_1=n_2\ \ \,{\rm mod}\ 2\\ m_1=m_2\ {\rm mod}\ 2\end{array}\right.
\end{array}
$$\smallskip
\end{lem}

\begin{proof}
The group $G_\v$ is the fundamental group of the non-orientable 3-manifolds $\M_\v$ which is two covered by a possibly punctured $\TT^2\times\I$; the cover involution extends to an orientation reversing  involution $\s$ of $\TT^2\times\I$ with at most isolated fixed points. It is known (cf. details in \cite{luft}) that there are up to isotopy 5 involutions with at most isolated fixed points on $\TT^2\times\I$, among which 3 are non-orientable. We set here $\I=[0,1]$ and $\SS^1=\R/2\pi\Z$.
According to \cite{kim}, up to isotopy $\s$ factors  as a product $\s((x,y),t)=(\phi(x,y),t)$ or $\s((x,y),t)=(\phi(x,y),1-t)$
for $\phi$ an homeomorphism of $\TT^2$.
There are up to isotopy 5  involutions of the torus. They are:\\
1. $\phi(x,y)=(x+\pi,y)$  with no fixed point and orbit space $\TT^2$,\\
2. $\phi(x,y)=(-x,y)$  with fixed point set $\SS^1\times \SS^0$ and orbit space $\SS^1\times \I$,\\
3. $\phi(x,y)=(y,x)$ with fixed point set a circle and orbit space a Moebius band,\\
4. $\phi(x,y)=(x+\pi,-y)$  with no fixed point and orbit space $\KB^2$,\\
5. $\phi(x,y)=(-x,-y)$  with fixed point set 4 points and orbit space $\SS^2$,\\
among them only 1 and 5 are orientation preserving. Since $\s$ is non-orientable and has at most isolated fixed points, the only 3 possibilities are:\\
(1). $\s(x,y,t)=(x+\pi,y,1-t)$;  $\M_\v=\TT^2\times\I/\s$ is the twisted $I$-bundle over the torus,  otherwise said the product of a Moebius band with $\SS^1$, and $G_\v$ admits the first presentation.\\
(2). $\s(x,y,t)=(x+\pi,-y,t)$; $\M_\v=\TT^2\times\I/\s$ is $\KB^2\times\I$ and $G_\v$ admits the second presentation.\\
(3). $\s(x,y,t)=(-x,-y,1-t)$; $\M_\v=\TT^2\times\I/\s$ has orientation covering space $\TT^2\times\I$ minus 4 balls centered on the fixed points, its boundary consists of four $\PP^2$ and one $\TT^2$, and $G_\v$ admits the third presentation. $\M_\v$ can be seen as: let $\PP$ be two copies of $\PP^2\times \I$ glued along two disks in their boundary, $\partial\PP$ consists of one $\KB^2$ and two $\PP^2$ and $\KB^2$ contains one annulus which is essential in $\PP$. Glue two copies of $\PP$ on this annulus to obtain $\M_\v$ (cf. details in \cite{luft}). \smallskip\\
The conjugacy criteria are  obtained by direct computations.
\end{proof}

Now several cases can occur, that one checks using the following lemma.
\begin{lem} 
One of the following cases occur:
\begin{itemize} 
\item[(a)] $Z_G(v^2)$ is conjugate to a subgroup of a Seifert vertex subgroup $G_\v$ of $G$.
\item[(b)] $Z_G(v^2)$ is conjugate to an edge subgroup $G_\e$ of $G$.
\item[(c)] $Z_G(v^2)$ is  conjugate neither into a Seifert vertex subgroup nor to an edge subgroup of $G$. Both $u$ and $v$ lie in  vertex subgroups containing $\ZZ$ as an index 2 subgroup, and do not lie in any edge subgroup. 
\end{itemize}
One can decide which case occurs.
\end{lem}

\begin{proof} 
That case (a), (b) or (c) occurs follows from Lemma \ref{squareroot} applied to $u$ and $v$ (when applying Lemma \ref{squareroot} to $u$, if $k\not\in h$  use the conjugating element $hk^{-1}u\in H$ rather than $hk^{-1}\not\in H$), since $Z_G(v^2)$ is generated by on the one hand $Z_H(v^2)$ and $v$ and on the other by $Z_H(v^2)$ and $u$. The algorithm $\mathfrak{ZP}(H)$ returns all vertex or edge subgroups of $H$ containing $Z_H(v^2)$. Using the algorithm in Lemma \ref{algopieces}.(iv) one finds all vertex and edge subgroups containing $v$, $u$. A vertex subgroup $G_\v$ contains $\ZZ$ as an index 2 subgroup if and only if  $H_\v$ is associated to a Seifert piece with among its Seifert invariants, has basis $\SS^1\times \I$ and no exceptional fiber.
\end{proof}

{\noindent\sl$\bullet$\ \begin{minipage}[t]{\largeur}
 Cases (a) and (b): $Z_G(v^2)$   is conjugate to a subgroup of a Seifert vertex subgroup $G_\v$ or to an edge subgroup $G_\e$ of $G$.\smallskip
\end{minipage}
 }

In such cases change $v$ into $hvh^{-1}$ and $u$ into $hk^{-1}ukh^{-1}$ so that $u,v$ both lie in $G_\v$ or $G_\e$ and $u^2=v^2$. \medskip

{\noindent\sl$\bullet$\ \begin{minipage}[t]{\largeur}
 Change $v$ and $u$ into their respective conjugate $hvh^{-1}$ and $hk^{-1}ukh^{-1}$ in $G_\v$, $G_\e$.\smallskip
\end{minipage}
 }

We now return to each of the cases (a), (b).\medskip

{\noindent\sl$\bullet$\ \begin{minipage}[t]{\largeur}
 Case (a): $Z_G(v^2)$   lies in a Seifert vertex subgroup $G_\v$ of $G$.\\
\end{minipage}
 }
\indent
One decides whether $u\sim v$ in $G$, using the following lemma and the solution to the conjugacy problem in $G_\v$ (Lemma \ref{dehnpbm}.(iv)).

\begin{lem}[In case $Z_G(v^2)$ is included in a Seifert vertex group] 
If $Z_G(v^2)$ is included in a Seifert subgroup $G_\v$ of $G$, then a solution to the conjugacy problem in $G_\v$ allows to decide whether $u\sim v$ in $G$.
\end{lem}

\begin{proof}
One has by hypothesis that $C_G(u^2,v^2)=Z_G(v^2)$ is included in $G_\v$. Since
  $C_G(u^2,v^2)\supset C_G(u,v)$, if $u$ and $v$ are conjugate in
  $G$, they must also be conjugate in $G_v$, that one can decide using $\mathfrak{CP}(G_\v)$ (Lemma \ref{dehnpbm}.(iv)).
\end{proof}

{\noindent\sl$\bullet$\ \begin{minipage}[t]{\largeur}
 Case (b): $Z_G(v^2)$   lies  in an edge subgroup $G_\e$ of $G$.\medskip
\end{minipage}
 }
\indent
In such case one decides whether $u\sim v$ using the following lemma.

\begin{lem}[Algorithm $\mathfrak{CP}(\Z\rtimes\Z)$]
If $Z_G(v^2)$ is included in an edge subgroup $G_\e$ of $G$, then one can  decide whether $u\sim v$ in $G$.
\end{lem}

\begin{proof}  Necessarily $G_\e=Z_G(v^2)$ is isomorphic to the group of the Klein bottle which has generators $t,b$ in $G$ with finite presentation 
$G_\e=<t,b\,|\,tbt^{-1}=b^{-1}>$. Set $a=t^2$ to obtain the alternative presentation $<a,b,t\,|\,[a,b]=1,\, t^2=a,\, b^t=b^{-1}>$; then use the algorithm in Lemma \ref{algopieces}.(iv)  to write $u$ and $v$ on the generators $a,b,t$, say $u=a^{n_1}b^{m_1}t$ and $v=a^{n_2}b^{m_2}t$. Then by Lemma \ref{virtualzz},  $u\sim v$ in $G_\v$ if and only if $n_1=n_2$ and $m_1=m_2 \mod 2$, and one can decide whether $u\sim v$ in $G_\e$. Since here again $C_G(u,v)=Z_G(v)\subset G_\e$, one has finally that $u\sim v$ in $G_\e$ if and only if $u\sim v$ in $G$.\smallskip
\end{proof}

{\noindent\sl$\bullet$\ \begin{minipage}[t]{\largeur}
 Case (c): $Z_G(v^2)$  is conjugate neither to a subgroup of  a vertex subgroup nor to an edge subgroup of $G$.\smallskip
\end{minipage}
 }

In that case both $u$ and $v$ lie in vertex subgroups $G_{\v_1}$, $G_{\v_2}$ of $G$ which both contain $\ZZ$ as an index 2-subgroup and do not lie in any edge subgroup. 
Now two cases can occur, according to whether $u,b$ both lie in a same vertex subgroup $\v_1=\v_2$ other not. One concludes in each case by using the following lemma as well as a solution to the word problem in $G_{\v_1}$.

\begin{lem}\label{vzz1} When $u,v$ both lie in vertex subgroup $G_{\v_1}$, $G_{\v_2}$ containing $\ZZ$ as index 2 subgroups and do not lie in any edge subgroup:
\begin{itemize}
\item[(i)] if $\v_1\not=\v_2$  then $u$ and $v$ are not conjugate in $G$,
\item[(ii)] if $\v_1=\v_2$, then $u$ and $v$ are conjugate in $G$ if and only if $u=v$.
\end{itemize}
\end{lem}

\begin{figure}[h]
\includegraphics[scale=0.9]{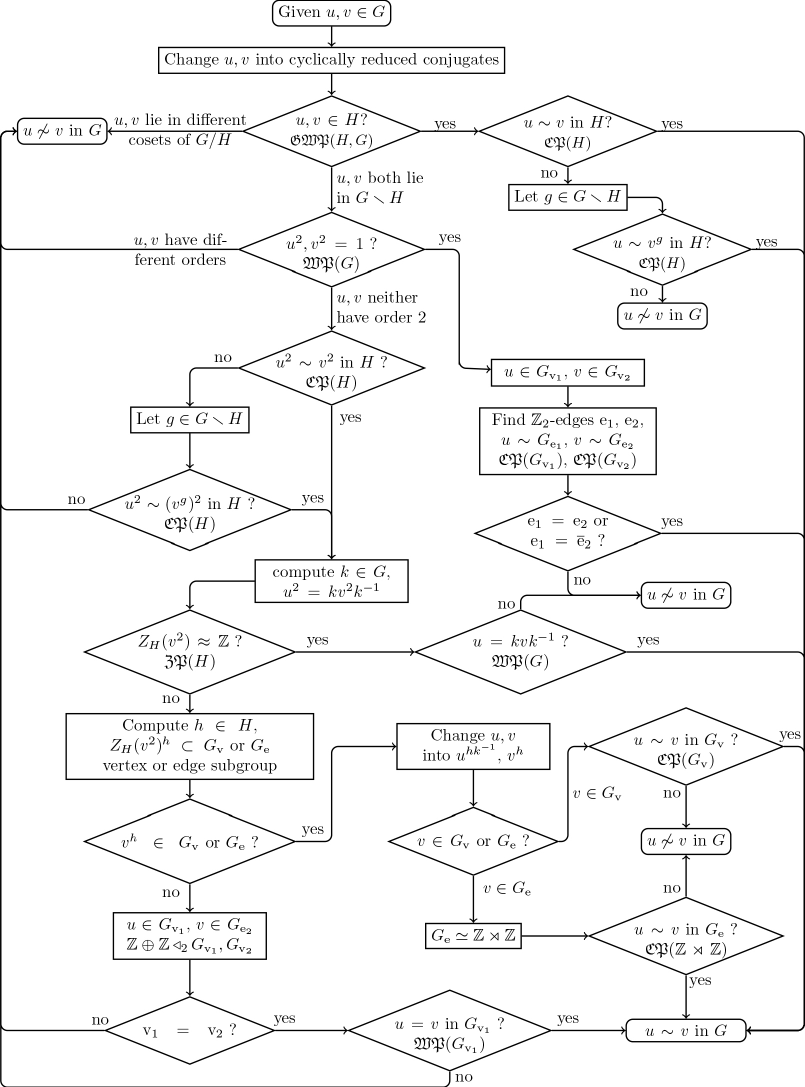}
\caption{The organigram of the algorithm}
\end{figure}

\begin{proof}
Case (2) of Lemma \ref{virtualzz} cannot occur, since $u,v$ do not lie in an edge subgroup (Lemma \ref{squareroot}), and neither can occur case (3) since $u$, $v$ have infinite order: case (1) of Lemma \ref{virtualzz} holds.

Denote by $G_{\e_1}$, $G_{\e_2}$ the respective $\ZZ$-edge subgroups of $G_{\v_1}$, $G_{\v_2}$ (cf. case (1) of Lemma \ref{virtualzz}). Since $G_{\e_i}$ is normal in $G_{\v_i}$, $i=1,2$,  by hypothesis $u$, $v$ are not conjugate in $G_{\v_i}$ to any element in $G_{\e_i}$, $i=1,2$. \smallskip\\
{\it Case (i)}. By deleting the pair of edges $\e_1, \ol{\e}_1$ in ${\rm M}$, $G$ splits as an amalgamated product $K\ast_{G_{\e_1}}G_{\v_1}$. The element $u$ lies in the right factor $G_{\v_1}$ while $v$ lies in the left factor $K$. Since $u$ is not conjugate in $G_{\v_1}$  to an element of $G_{\e_1}$, it follows from \cite[Theorem 4.6]{mks} that $u$ and $v$ are not conjugate in $G$.\smallskip\\
{\it Case (ii)}. Here, $u,v\in G_{\v_1}$ and the same argument as in (i) shows that $u$ and $v$ are conjugate in $G$ if and only if they are conjugate in $G_{\v_1}$, and since $G_{\v_1}$ is abelian, if and only if they are equal.
\end{proof}



\end{document}